\documentclass[10pt,letterpaper]{article}
\usepackage[T1]{fontenc}
\usepackage{lmodern,microtype}
\pdfmapfile{+euler.map}
\pdfmapfile{+rsfs.map}
\usepackage[margin=1in]{geometry}
\usepackage{amsmath,amssymb,amsthm,mathtools,mathrsfs,booktabs,array}
\usepackage{xcolor}
\usepackage[hidelinks]{hyperref}
\allowdisplaybreaks[2]
\newtheorem{theorem}{Theorem}[section]
\newtheorem{lemma}[theorem]{Lemma}
\newtheorem{conjecture}[theorem]{Conjecture}

\theoremstyle{remark}
\numberwithin{equation}{section}
\newcommand{\LIS}{\operatorname{LIS}}
\newcommand{\PP}{\mathbb P}
\newcommand{\Ai}{\operatorname{Ai}}

\newcommand{\diag}{\operatorname{diag}}
\newcommand{\Ree}{\operatorname{Re}}
\newcommand{\Ime}{\operatorname{Im}}
\newcommand{\e}{\mathrm e}

\hypersetup{pdftitle={On Chen’s conjecture  the log-concavity of the longest increasing subsequence distribution},pdfauthor={},pdfsubject={On Chen’s conjecture for the log-concavity of the longest increasing subsequence distribution}}
\title{ On Chen’s conjecture for  the log-concavity of the longest increasing subsequence distribution }
\author{Ernest X. W. Xia\\[6pt]
{\small School of Mathematical Sciences,}
 \\
 {\small Suzhou University of Science and
	Technology,}
	 \\ 
{\small Suzhou 215009, Jiangsu, P. R. China,}\\
 {\small Email:	ernestxwxia@163.com}
}

\date{}
\begin{document}
\maketitle

	\noindent\textbf{Abstract:}
	Let $P_{n,k}$ count permutations of $[n]=\{1,\ldots,n\}$ whose longest
	increasing subsequence has length $k$. Chen conjectured that
	$(P_{n,k})_{k=1}^n$ is log-concave for every positive integer $n$.
 In this paper, we prove that Chen's conjecture is true 
   when $n\ge 10^{10^{65}}$. The proof combines
Gessel's determinantal formula with quantitative coefficient extraction,
Tracy--Widom density estimates, and exact inequalities for the extreme
lengths.



\section{Introduction}
\label{sec:introduction}

Let $\mathfrak S_n$ be the symmetric group on $[n]=\{1,\ldots,n\}$.
For a permutation $w=w_1\cdots w_n$, an increasing subsequence is
$w_{i_1}\cdots w_{i_k}$ with $i_1<\cdots<i_k$ and
$w_{i_1}<\cdots<w_{i_k}$. Its maximum length is denoted by $\LIS(w)$.
For example, $2358$ and $2356$ are longest increasing subsequences of
$2735861$, so $\LIS(2735861)=4$. Define
\begin{equation}\label{eq:counts}
 P_{n,k}:=\#\{w\in\mathfrak S_n:\LIS(w)=k\},\qquad
 p_{n,k}:=\frac{P_{n,k}}{n!},
\end{equation}
and set both quantities equal to zero outside $1\le k\le n$.

A nonnegative real sequence $a_1,\ldots,a_n$ is log-concave if
$a_k^2\ge a_{k-1}a_{k+1}$ for $1\le k\le n$, with
$a_0=a_{n+1}=0$. Such inequalities occur throughout algebra,
combinatorics and geometry; see Stanley~\cite{StanleyLC} and
Brenti~\cite{Brenti}.
Chen proposed the following conjecture~\cite[Conjecture~1.1]{Chen}
 on the log-concavity of $P_{n,k}$.

\begin{conjecture}[Chen]\label{conj:chen}
For every integer $n\ge1$ and every $1\le k\le n$,
\begin{equation*}
 P_{n,k}^2\ge P_{n,k-1}P_{n,k+1}.
\end{equation*}
\end{conjecture}

A partition of a positive integer \(n\) is a finite
nonincreasing sequence
$
\lambda=(\lambda_1,\lambda_2,\ldots,\lambda_{\ell})
$
of positive integers such that
$
\lambda_1+\lambda_2+\cdots+\lambda_{\ell}=n.
$
  We write \(\lambda\vdash n\)
to indicate that \(\lambda\) is a partition of \(n\) \cite[Chapter~1]{AndrewsBook}.
The Robinson--Schensted correspondence~\cite{Schensted} provides a natural formulation
of this problem in terms of Young tableaux. For a partition
$\lambda$, let $f^\lambda$ denote the number of standard Young
tableaux of shape $\lambda$. Schensted's theorem  gives
\begin{equation}\label{eq:rsk-count}
 P_{n,k}=\sum_{\substack{\lambda\vdash n\\\lambda_1=k}}(f^\lambda)^2.
\end{equation}

Gessel~\cite{Gessel} obtained a determinantal generating function for
permutations whose longest increasing subsequence has bounded length.
His formula expresses the generating function of the cumulative counts
as a Toeplitz determinant with modified Bessel entries; its precise
normalization is recorded in~\eqref{pre:gessel}. Taking successive
differences recovers the exact counts in~\eqref{eq:counts}. Chen ~\cite[Section~2]{Chen} used
Gessel's formula to compute these counts and verify the conjecture for
$n\le18$. The same identity is the starting point
of the analytic argument below. For the associated Poissonized
Plancherel measure, Borodin, Okounkov and Olshanski~\cite{BOO}
subsequently obtained the discrete Bessel kernel representation used
in our determinant estimates.

B\'ona, Lackner and Sagan~\cite{BLS} verified Chen's conjecture for
$n\le50$ and proved analogous inequalities for permutations whose
Robinson--Schensted shapes are hooks or have at most two rows. They also
formulated a companion conjecture for involutions and related the two
problems through the Robinson--Schensted correspondence. Gao, Xie and
Yang~\cite{GXY} extended these restricted-shape results and established
Schur-positive refinements. Gao, Xie and
Yang~\cite{GXY} further proposed Schur-positivity
conjectures which imply the respective log-concavity
conjectures for permutations and involutions.   These results support Chen's conjecture,
but do not establish the inequality for all shapes at a fixed size.

An analytic approach is suggested by the theorem of Baik, Deift and
Johansson~\cite{BDJ}: after centering at $2\sqrt n$ and scaling by
$n^{1/6}$, the longest increasing subsequence length of a uniform
permutation converges in distribution to the GUE Tracy--Widom law
\cite{TW}. B\'ona, Lackner and Sagan~\cite[Section~5]{BLS} presented
Deift's proof that the density of this limiting law is log-concave on
the nonnegative half-line. Baslingker, Krishnapur and Madiman~\cite{BKM}
proved log-concavity of the Tracy--Widom densities on the whole real
line.  
They~\cite[Theorems~1--2]{BKM} also
obtained two results directly related to Chen's conjecture. First, every
row-length marginal under Poissonized Plancherel measure is log-concave.
Second, for each fixed nonnegative integer $j$, Chen's inequality holds
for $k\in\{n-j,\ldots,n\}$ once $n$ is sufficiently large, with a
threshold depending on $j$. The first result concerns a mixture over
partition sizes, and the second concerns a fixed distance from the
right endpoint. Neither yields a threshold valid for all $k$ at exact
size $n$; in particular, log-concavity is not automatically preserved
by removing Poissonization.

The purpose of this paper is to prove that Chen's conjecture is true 
 when $n\geq 10^{10^{65}}$. 

\begin{theorem}\label{thm:main}
Set $N:=10^{10^{65}}$. For every integer $n\ge N$ and every
$1\le k\le n$,
\begin{equation*}
 P_{n,k}^{2}>P_{n,k-1}P_{n,k+1}.
\end{equation*}
\end{theorem}

The proof of Theorem \ref{thm:main} estimates the discrete curvature of $\log p_{n,k}$ in five
overlapping regions. Near the typical length, quantitative approximation
by the Tracy--Widom density is combined with coefficient extraction at
fixed size. The remaining regions are treated by Toeplitz determinant
estimates and exact tableau identities. The threshold is deliberately
conservative: the objective is an explicit bound independent of $k$,
and the theorem does not settle Chen's conjecture for every $n<N$.

The paper is organized as follows. Section~\ref{sec:preliminaries}
introduces the notation, records Gessel's formula and the kernel
identities, and establishes the quantitative estimates used in the
proof. Section~\ref{sec:proof} proves the five regional inequalities
and combines them to establish Theorem~\ref{thm:main}.
Section~\ref{sec:concluding} discusses the remaining finite-size problem
and possible improvements of the threshold.

\section{Preliminaries}\label{sec:preliminaries}

Choose a permutation uniformly from $\mathfrak S_n$, so that
each of the $n!$ permutations has probability $1/n!$.
Let $L_n$ denote the length of its longest increasing subsequence,
and let $\PP$ denote probability under this uniform distribution.
For $1\le k\le n$, define the cumulative probability
\begin{equation*}
	a_{n,k}:=\PP\{L_n\le k\}
	=\frac{1}{n!}\sum_{j=1}^{k}P_{n,j}.
\end{equation*}
For $n\ge1$, we set  $a_{n,0}:=0$ and  $a_{n,k}:=1$ for $k\ge n$.
We also set $a_{0,k}:=1$ for $k\ge0$.
Thus $a_{n,k}$ is the probability that the longest increasing
subsequence has length at most $k$. In the notation
of~\eqref{eq:counts},
$$
	p_{n,k}=\PP\{L_n=k\}=a_{n,k}-a_{n,k-1}.
$$
We also write
\begin{equation*}
	\mathcal C_{n,k}:=2\log p_{n,k}
	-\log p_{n,k-1}-\log p_{n,k+1}
	\qquad(2\le k\le n-1).
\end{equation*}
At interior indices, $\mathcal C_{n,k}>0$ is equivalent to the desired
strict inequality. Define the scaled coordinates by
\begin{equation}\label{pre:coordinates}
 t_k(n):=\frac{k-2\sqrt n}{n^{1/6}},\qquad
 s_{n,k}:=t_k(n)-\frac1{2n^{1/6}},\qquad T:=-t_k(n).
\end{equation}
We abbreviate $s:=s_{n,k}$ when $n,k$ are fixed. The half-lattice shift
centers each exact mass between consecutive cumulative probabilities.
All logarithms are natural.

 For a partition $\lambda=(\lambda_1,\lambda_2,\ldots, \lambda_k)$,  its Young diagram has $\lambda_i$ cells in row
$i$. A standard Young tableau fills these cells with $1,\ldots,n$,
increasing along rows and columns. For a cell $u$, its hook length
$h_\lambda(u)$ counts $u$, the cells to its right, and those below it.
The hook-length formula of Frame, Robinson and Thrall~\cite{FRT} is
\begin{equation*}
 f^\lambda=\frac{n!}{\prod_{u\in\lambda}h_\lambda(u)}.
\end{equation*}
Together with~\eqref{eq:rsk-count}, this gives the exact tableau
enumeration used below.

We use standard trace-ideal and Cauchy singular-integral results, with
sources specified where they enter. The subsequent estimates keep track of the contour geometry, the
moving index, and the constants required for the explicit threshold.

We use Fredholm determinants on the complex Hilbert spaces $L^2(X,\mu)$.
These consist of measurable functions, identified when equal almost
everywhere, for which
\[
 \|f\|_{L^2(X,\mu)}:=\left(\int_X|f(x)|^2\,d\mu(x)\right)^{1/2}<\infty.
\]
Here $X$ is a half-line, a countable set, or a finite disjoint union of
these, with Lebesgue or counting measure. On a vertical contour the
measure is $dy/(2\pi)$. Write $\|\cdot\|_1$ and $\|\cdot\|_2$ for the
trace and Hilbert--Schmidt norms: for singular values $s_j(A)$, these
are $\|A\|_1:=\sum_j s_j(A)$ and
$\|A\|_2:=(\sum_j s_j(A)^2)^{1/2}$; finiteness defines the respective
classes. Thus a square-integrable kernel $H$
satisfies
\[\|H\|_2^2=\iint_{X\times X}|H(x,y)|^2\,d\mu(x)d\mu(y),\]
 with sums for
matrices \cite[Theorem~3.3.2 and Remark~3.3.1]{Kostenko}.
Products are operator compositions, without conjugating the
kernel; in particular $H^2$ need not be positive for complex parameters.

Throughout, $I$ denotes the identity operator on the underlying Hilbert
space, and $I_2$ denotes the $2\times2$ identity matrix.
We use the standard trace-ideal bounds
\cite[\S\S2--4, especially (4.1)]{BornemannDet}:
\begin{equation}\label{pre:fredholmbounds}
 \begin{gathered}
 \|AB\|_1\le\|A\|_2\|B\|_2,\qquad
 |\det(I+V)|\le e^{\|V\|_1},\\
 |\det(I+V)-\det(I+W)|
 \le\|V-W\|_1e^{1+\max(\|V\|_1,\|W\|_1)}.
 \end{gathered}
\end{equation}
Here $A,B$ are Hilbert--Schmidt and $V,W$ are trace class. In particular,
\begin{equation}\label{pre:HSsquare}
 \begin{gathered}
 \|H^2-L^2\|_1\le(\|H\|_2+\|L\|_2)\|H-L\|_2,\\
 |\det(I-H^2)-\det(I-L^2)|
 \le(\|H\|_2+\|L\|_2)\|H-L\|_2
       e^{1+\max(\|H\|_2^2,\|L\|_2^2)}.
 \end{gathered}
\end{equation}
The first inequality follows by factoring
$H^2-L^2=H(H-L)+(H-L)L$.
The exterior-power expansion also gives
\[
|\det(I+V)-1|\le e^{\|V\|_1}-1;
\]
 see \cite[Proposition~3.4.3 and (3.4.15)]{Kostenko}.

The kernel families below are pointwise holomorphic and locally dominated
by a square-integrable envelope. They are therefore holomorphic in
Hilbert--Schmidt norm: scalar pairings are holomorphic by dominated
integration, and weak and norm holomorphy coincide
\cite[Theorem~B.0.1]{Kostenko}. Products are holomorphic in trace norm.
Finite-rank compression and \eqref{pre:fredholmbounds} then show that
their Fredholm determinants are locally uniform limits of holomorphic
finite determinants. This also justifies Cauchy estimates in either
parameter.

For the determinants used below, $V=-H^2$ with $H$ Hilbert--Schmidt.
We use the composition kernel
\[
 V(x,y):=-\int_X H(x,z)H(z,y)\,d\mu(z),
\]
including on the diagonal, at points where the row and column functions
belong to $L^2(X,\mu)$; exceptional null sets are assigned zero.
The Gram determinant inequality followed by Cauchy--Schwarz gives
\[
 \frac1{m!}\int_{X^m}\bigl|\det[V(x_i,x_j)]_{i,j=1}^m\bigr|
 \,d\mu(x_1)\cdots d\mu(x_m)
 \le\frac{\|H\|_2^{2m}}{m!}.
\]
Indeed, the determinant is bounded by the product of the square roots
of the two Gram determinants formed from the rows $H(x_i,\cdot)$
and the conjugated columns $\overline{H(\cdot,x_i)}$; Hadamard's
inequality and Cauchy--Schwarz then give the displayed integrated bound.
Consequently the Fredholm series is absolutely convergent and
\[
 \det(I+V)=\sum_{m=0}^{\infty}\frac1{m!}
 \int_{X^m}\det[V(x_i,x_j)]_{i,j=1}^m
 \,d\mu(x_1)\cdots d\mu(x_m),
\]
with zeroth term one. To justify the equality, approximate $H$ by
finite-rank kernels in Hilbert--Schmidt norm. Their squares converge
in trace norm by~\eqref{pre:HSsquare}; each integral coefficient
converges by multilinearity and the same Gram determinant and
Cauchy--Schwarz bounds. The displayed summable bound permits passage
to the whole series. The finite-rank identity therefore yields the
operator determinant, consistently with~\cite[\S3]{BornemannDet}.

For $k\ge0$, define
\begin{equation}\label{pre:det-def}
 \begin{gathered}
 D_k(t):=\det\bigl(I_{i-j}(2t)\bigr)_{i,j=1}^{k},\quad D_0(t):=1,\quad 
 F_k(z):=D_k(\sqrt z),\quad Q_k(z):=e^{-z}F_k(z),
 \end{gathered}
\end{equation}
where $I_\nu$ is the modified Bessel function. The square-root expressions
are single-valued because $D_k$ is even. Gessel's determinantal
identity~\cite{Gessel} gives
\begin{equation}\label{pre:gessel}
 F_k(z)=\sum_{m\ge0}a_{m,k}\frac{z^m}{m!}.
\end{equation}
For real $z>0$, $Q_k(z)$ is therefore the cumulative distribution
under Poissonized Plancherel measure of mean $z$. The discrete Bessel
kernel representation of Borodin, Okounkov and Olshanski~\cite{BOO}
gives the separate identity
\begin{equation}\label{pre:bessel}
 Q_k(z)=\det(I-H_k(z)^2),\qquad
 (H_k(z))_{ij}:=J_{k+i+j+1}(2\sqrt z),\quad i,j\ge0.
\end{equation}
For a power series $H(z)=\sum_{m\ge0}c_mz^m$, the notation
$[z^n]H(z)$ denotes the coefficient $c_n$ of $z^n$.
For $k\ge1$, put
\begin{equation}\label{pre:massgen}
 \begin{gathered}
 g_k(z):=F_k(z)-F_{k-1}(z),\quad 
 p_{n,k}=n![z^n]g_k(z),\quad P_{n,k}=(n!)^2[z^n]g_k(z).
 \end{gathered}
\end{equation}
On a positive-centered complex disk we use the principal square root.
The identity in \eqref{pre:bessel}, initially obtained for positive real
parameters, extends holomorphically: the Bessel series gives
$|J_m(2t)|\le e^{R^2}R^m/m!$ for $|t|\le R$, so the Hankel family is
locally holomorphic in Hilbert--Schmidt norm and its square in trace norm.
The Fredholm determinant is therefore holomorphic; the identity theorem
applies. The Bessel orders remain integers throughout.

For $z$ with the principal square root,
\begin{equation}\label{pre:haarbound}
 |Q_k(z)|\le \exp\!\left(\frac{|z|-\Re z}{2}\right).
\end{equation}
Consequently, for $|\theta|\le\pi$,
\begin{equation}\label{pre:outerhaar}
 e^{-n}\left|e^{ne^{i\theta}}Q_k(ne^{i\theta})\right|
 \le e^{-n\sin^2(\theta/2)}\le e^{-n\theta^2/\pi^2}.
\end{equation}

The Fourier integral for each Bessel entry \cite[(10.32.3)]{DLMF} and expansion of two
finite determinants give the multiple-integral form of the unitary-group
identity~\cite{Rains}
\begin{equation}\label{pre:angular-integral}
 D_k(t)=\frac1{k!(2\pi)^k}\int_{[-\pi,\pi]^k}
 e^{2t\sum_{j=1}^k\cos\theta_j}
 \prod_{1\le i<j\le k}|e^{i\theta_i}-e^{i\theta_j}|^2
 \,d\theta_1\cdots d\theta_k.
\end{equation}
The nonnegative weight gives $|D_k(t)|\le D_k(\Re t)$. For $u\ge0$, \eqref{pre:gessel} and
$0\le a_{m,k}\le1$ imply $D_k(u)\le e^{u^2}$.
Using $(\Re\sqrt z)^2=(|z|+\Re z)/2$ proves \eqref{pre:haarbound}. Multiplication by $|e^z|$ at $z=ne^{i\theta}$ proves the first inequality in \eqref{pre:outerhaar}; the second follows from $\sin(|\theta|/2)\ge|\theta|/\pi$.

Define
\begin{equation*}
 b_j(n):=[w^j]e^{-nw}(1+w)^n.
\end{equation*}
Then
\begin{equation}\label{pre:charlierbounds}
 b_0=1,\qquad b_1=0,\qquad b_2=-n/2,\qquad
 |b_j(n)|\le n^{\lfloor j/2\rfloor}\quad(n\ge1).
\end{equation}
For every polynomial $P(z)$,
\begin{equation}\label{pre:taylorcoeff}
 n![z^n]e^zP(z)=\sum_{j\ge0}b_j(n)P^{(j)}(n).
\end{equation}
The normalized Cauchy coefficient kernel on $|z|=n$ has absolute integral less than $2$:
\begin{equation}\label{pre:cauchyLone}
 \frac{n!e^n}{2\pi n^n}\int_{-\pi}^{\pi}
          e^{n(\cos\theta-1)}\,d\theta<2.
\end{equation}

The generating identity
$n![z^n]e^z e^{w(z-n)}=e^{-nw}(1+w)^n$
proves \eqref{pre:taylorcoeff} by comparing coefficients in $w$ and Taylor expanding $P$. Since
\[
 e^{-nw}(1+w)^n
 =\exp\!\left(n\sum_{\ell\ge2}\frac{(-1)^{\ell-1}w^\ell}{\ell}\right),
\]
$b_j(n)$ has degree at most $\lfloor j/2\rfloor$ in $n$. The sum of absolute values of its coefficients is at most
\[
 [w^j]\exp\!\left(\sum_{\ell\ge2}\frac{w^\ell}{\ell}\right)
 =[w^j]\frac{e^{-w}}{1-w}\le1.
\]
This proves \eqref{pre:charlierbounds}. Finally use
$1-\cos\theta\ge2\theta^2/\pi^2$ for $|\theta|\le\pi$,
the Gaussian integral over the real line, and the explicit Stirling upper bound
$n!e^n/n^n\le\sqrt{2\pi n}\,e^{1/(12n)}$.
The left side of \eqref{pre:cauchyLone} is at most
$e^{1/12}\pi/2<2$.

Let $f$ be holomorphic on a disk $|z-x_0|<R$ and continuous on its closure. Suppose $|f|\le A$ on the real diameter and $|f|\le B$ on the disk, where $0<A\le B$. Then
\begin{equation}\label{pre:twoconstants}
 |f(z)|\le A^{1/3}B^{2/3}\qquad (|z-x_0|\le R/2).
\end{equation}

After translation and scaling, consider the upper half of the unit disk. The harmonic measure $\omega$ of its diameter is
\[
 \omega(z):=1-\frac2\pi\arg\frac{1+z}{1-z}.
\]
For $|z|\le1/2$, the fractional-linear map sends the disk into the disk with center $5/3$ and radius $4/3$. Hence its argument is at most
$\arcsin(4/5)=2\arctan(1/2)<\pi/3$, so $\omega(z)>1/3$.
The maximum principle for the subharmonic function $\log|f|$ gives
\[
\log|f(z)|\le\omega\log A+(1-\omega)\log B.
\]
The lower half-disk is identical, and the bound extends to the diameter by continuity.

The centered second-difference instance of the Peano-kernel formula
\cite[\S9, (47)--(48), p.~62]{DeBoor} states that, for
a twice continuously differentiable function $v$ on $[x-h,x+h]$,
\begin{equation}\label{pre:peano}
 2v(x)-v(x-h)-v(x+h)
 =-\int_{-h}^{h}(h-|y|)v''(x+y)\,dy.
\end{equation}
Let $b,f\in C^2(J)$ be positive on the same interval $J$, and suppose
that, pointwise on $J$, $f\ge m$, $|f'|,|f''|\le L$, $L\ge1$, and
$|b^{(j)}-f^{(j)}|\le\varepsilon$ for $0\le j\le2$, where
$0<m\le1$ and $0\le\varepsilon\le\min(m/2,1)$. Then
\begin{equation}\label{pre:logperturbation}
 \left|(\log b)''-(\log f)''\right|
 \le2\varepsilon(m^{-1}+Lm^{-2})(1+5L/m).
\end{equation}

For \eqref{pre:logperturbation}, use $b\ge m/2$ and subtract
$b''/b-(b'/b)^2$ and $f''/f-(f'/f)^2$.

Let $q$ be the positive Hastings--McLeod solution of
$q''=xq+2q^3$, with $q(x)\sim\Ai(x)$ at $+\infty$ and
$q(x)\sim\sqrt{-x/2}$ at $-\infty$. Write
\begin{equation}\label{pre:TW-def}
 \begin{gathered}
 u(x):=\int_x^\infty q(t)^2\,dt,\quad
 F(x):=\exp\!\left(-\int_x^\infty(t-x)q(t)^2\,dt\right),\quad 
 f(x):=F'(x)=F(x)u(x).
 \end{gathered}
\end{equation}
The Airy operator
\begin{equation*}
 C_x(r,s):=\int_0^\infty\Ai(r+v)\Ai(s+v)\,dv
 \quad\text{on }L^2(x,\infty)
\end{equation*}
is positive and trace class. After translation it equals $A_x^2$, where
$A_x(r,s):=\Ai(x+r+s)$ on $L^2(0,\infty)$. The Tracy--Widom
identities~\cite{TW} give
\begin{equation}\label{pre:TW-det}
 F(x)=\det(I-C_x)=\det(I-A_x^2).
\end{equation}

The following identity is the integrated form of the calculation in
\cite[Appendix~A, proof of Lemma~2, (36)--(40), pp.~23--24]{BKM};
the positivity of $q$ is recorded in Lemma~3 there.

For every real $x$,
\begin{equation}\label{pre:positivecurvature}
 \kappa(x):=-(\log f(x))''
 =\frac{2q(x)^2}{u(x)^2}\int_x^\infty\frac{u(t)^2}{q(t)^2}\,dt>0.
\end{equation}

Set
\begin{equation}\label{pre:airy-diagonal}
 \begin{gathered}
 K(x):=\int_x^\infty\Ai(y)^2\,dy=\Ai'(x)^2-x\Ai(x)^2,\quad 
 \tau(x):=\int_x^\infty(y-x)\Ai(y)^2\,dy.
 \end{gathered}
\end{equation}
Thus $C_x(x,x)=K(x)$ and $\operatorname{tr}C_x=\tau(x)$. We shall use
\begin{equation*}
 \begin{gathered}
 0<\Ai(y)<\tfrac25e^{-y}\quad(y\ge0),\qquad 0<-\Ai'(0)<1,\\
 \Ai(y)>e^{-9}\quad(-1\le y\le2),\qquad
 K(0)<\tfrac2{25},\quad\tau(0)<\tfrac1{25},\quad\tau(-1)<\tfrac5{12},
 \end{gathered}
\end{equation*}
and
\begin{equation*}
 \Ai(y)\ge e^{-2y^2}\qquad(y\ge2).
\end{equation*}
Here are analytic estimates establishing all these constants. The positive
Bessel representation \cite[\S\S\,9.6, 10.32]{DLMF}, restricted to
$0\le v\le\sqrt{3/2}\,y^{-3/4}$, gives
\[
 \Ai(y)=\frac{\sqrt y}{\pi\sqrt3}
 \int_0^\infty e^{-(2/3)y^{3/2}\cosh v}\cosh(v/3)\,dv
 \ge\frac{e^{-(2/3)y^{3/2}-1}y^{-1/4}}{\pi\sqrt2}
 >e^{-2y^2}\qquad(y\ge2).
\]
On the Airy contour $w=1+it$,
\[
 |\Ai(y)|\le\frac{e^{1/3-y}}{2\sqrt\pi}<\tfrac25e^{-y},
 \qquad
 |\Ai'(0)|\le\frac{e^{1/3}}{2\pi}(\sqrt\pi+1)<1.
\]
The positive integral and Airy equation show that $\Ai>0$ and $\Ai'<0$
on $[0,\infty)$. Up to a first zero on $[-1,0]$, the same equation gives,
for $0\le t\le1$,
\[
 \Ai(0)(1-t^3/6)-\Ai'(0)(t-t^4/12)
 \le\Ai(-t)\le\Ai(0)-t\Ai'(0).
\]
The positive lower bound rules out such a zero. Since
$\Ai(0)\ge\Ai(2)\ge e^{-8}$, it also proves the claimed lower bound
on $[-1,2]$. Integrating the upper estimates gives the bounds for $K(0)$
and $\tau(0)$, as well as
\[
 \tau(-1)\le
 \int_0^1(1-t)\bigl(\Ai(0)-t\Ai'(0)\bigr)^2\,dt+K(0)+\tau(0)
 <\tfrac5{12}.
\]

\begin{lemma}\label{pre:TWcompact}
For $T\ge3$ and $|x|\le T$,
\begin{equation*}
 f(x)\ge e^{-10(T+2)^3},\qquad
 \kappa(x)\ge e^{-20(T+2)^2}.
\end{equation*}
\end{lemma}
\begin{proof}
For $x\ge-1$, the Airy kernel is pointwise nonnegative and
$\|C_x\|\le\operatorname{tr}C_x=\tau(x)<5/12$.
The Neumann expansion of $R_x=C_x(I-C_x)^{-1}$ and
$|C_x(r,s)|\le\sqrt{K(r)K(s)}$ give
\[
 0\le R_x(r,s)\le\frac{\sqrt{K(r)K(s)}}{1-\tau(x)}.
\]
Using $u(x)=R_x(x,x)$ and
$q(x)=\Ai(x)+\int_x^\infty R_x(x,y)\Ai(y)\,dy$ \cite{TW}, and the
spectral product for a positive trace-class operator \cite[\S\S\,2--3]{BornemannDet}, we obtain
\begin{equation*}
 q(x)\ge\Ai(x),\qquad
 K(x)\le u(x)\le\frac{K(x)}{1-\tau(x)},\qquad
 F(x)\ge1-\tau(x).
\end{equation*}
Indeed, the last inequality follows from
$\prod_j(1-\lambda_j)\ge1-\sum_j\lambda_j$ for the nonnegative
eigenvalues of $C_x$. Cauchy--Schwarz also gives
\begin{equation*}
 0\le q(x)-\Ai(x)
 \le\frac{\sqrt{K(x)}}{1-\tau(x)}
       \int_x^\infty\sqrt{K(y)}\Ai(y)\,dy
 \le\frac{K(x)\sqrt{\tau(x)}}{1-\tau(x)}.
\end{equation*}
Consequently $q(0)<5/12$, $u(0)<1/12$, $F(0)>24/25$, and $q(x)<1$
for $x\ge0$.

By the monotonicity of $q$ \cite[Appendix~A, Lemma~3, p.~24]{BKM},
$q(x)\ge q(2)\ge\Ai(2)>e^{-9}$ for $x\le2$.
Comparison with $\sqrt{1-x}$ gives
\begin{equation*}
 q(x)\le\sqrt{1-x}\qquad(x\le0).
\end{equation*}
The difference is negative at $0$ and near $-\infty$; a positive
interior maximum contradicts the differential equation, since
$xy+2y^3$ is increasing for $y\ge\sqrt{1-x}$ and is positive there.
Integration yields
\begin{equation*}
 u(-T)\le T^2/2+T+1,\qquad
 F(-T)\ge\exp(-T^3/6-T^2/2-T-1).
\end{equation*}
It follows that, uniformly for $|x|\le T$,
\[
 f(x)\ge F(-T)\int_T^{T+1}\Ai(y)^2\,dy
 \ge e^{-T^3/6-T^2/2-T-1-4(T+1)^2}
 \ge e^{-10(T+2)^3}.
\]
Also $q(x)^2\ge e^{-4T^2}$ and $u(x)\le(T+1)^2$. On $[T,T+1]$,
$q\le1$ and $u\ge\int_{T+1}^{T+2}\Ai(v)^2\,dv\ge e^{-4(T+2)^2}$.
Substitution in \eqref{pre:positivecurvature} proves
\[
 \kappa(x)\ge\frac{2}{(T+1)^4}e^{-4T^2-8(T+2)^2}
 \ge e^{-20(T+2)^2}.
\]
\end{proof}

We work in $L^p(\Sigma;\mathbb C^{2\times2})$ with arclength measure
$|ds|$ and the Frobenius matrix norm
$$\|A\|_{\mathrm F}:=(\sum_{i,j=1}^2|A_{ij}|^2)^{1/2}.$$
  The corresponding
function norms are denoted by $\|\cdot\|_p$. Thus
$\|f\|_2^2=\int_\Sigma\|f(s)\|_{\mathrm F}^2|ds|$ for matrix-valued functions.
In particular, $\|AB\|_{\mathrm F}\leq\|A\|_{\mathrm F}\|B\|_{\mathrm F}$.
All constants below are uniform on the stated complex parameter disk.

We next obtain a quantitative estimate for the determinants
in~\eqref{pre:det-def}. For real $t>0$, let $\pi_k(z;t)$ be the monic orthogonal polynomial for
$e^{t(z+z^{-1})}dz/(2\pi iz)$ on the unit circle.  Its constant term and the ratio $D_{k-1}/D_k$ extend
meromorphically by the moment determinant formulas.

\begin{lemma}\label{rh:quantitative}
Let $t_0>0$, and put
\begin{equation*}
 \Delta_0:=2t_0-k=kb_0^2>0,\qquad
 0<b_0\leq10^{-6},\qquad \Lambda_0:=kb_0^3\geq10^{30}.
\end{equation*}
On the closed disk $|t-t_0|\leq\Delta_0/10^4$, the functions
$D_k,D_{k-1}$ have no zeros and $\pi_k(0;t)$ is holomorphic in a
neighborhood of the disk.  Choose
\[
 b^2:=\frac{2t-k}{k},\quad b(t_0)=b_0,\quad
 \gamma:=1+b^2=\frac{2t}{k},\quad \Delta:=2t-k,
\]
and choose logarithms by their real values at $t_0$.  Then, with
$A_k(t):=\log(D_k(t)/D_{k-1}(t))$,
\begin{align}
 \left|\pi_k(0;t)^2-\left(1-\frac{k}{2t}\right)\right|
 &\leq \frac{10^{26}}{\Delta_0^2},\label{rh:reflection}\\
 |A_k(t)-A_k^0(t)|
 &\leq 10^{25}\frac{b_0}{\Lambda_0^2},\label{rh:norm}\\
 A_k^0(t)
 &:=\Delta-\left(k-\frac12\right)\log\gamma
       +\frac1{8\Delta}+\frac1{12k}.\notag
\end{align}
The logarithmic remainder in \eqref{rh:norm} is holomorphic.
\end{lemma}

\begin{proof}
We first prove the auxiliary Lemmas~\ref{rh:phasemargins},
\ref{rh:airybound}, \ref{rh:fixedmatching}, \ref{rh:cauchy}, and
\ref{rh:jumpbounds}, then use them to complete the reconstruction
and zero-free estimates for Lemma~\ref{rh:quantitative}.
The exact transformations are those of \cite[Section~6]{BDJ} and
\cite[(52)--(64)]{BBD}.  We retain the geometric and jump estimates
needed for an explicit constant, and use bounded Cauchy operators for
the reconstruction.

In the notation of \cite{BBD}, $q=k$, $\gamma=2t/k$, $C_1$ is the
short gap, and $C_2$ is the long band through $1$.  Set
\begin{equation*}
 z=M(x):=\frac{x+i}{x-i},\qquad x=i\frac{z+1}{z-1},\qquad
 \xi_\pm:=M(\pm b).
\end{equation*}
The fourth root $\beta:=((z-\xi_+)/(z-\xi_-))^{1/4}$ is normalized to
$1$ at infinity and cut along the long band.  The outer model is
\begin{equation*}
 C:=\begin{pmatrix}1&-1\\-i&-i\end{pmatrix},\qquad
 \sigma_3:=\begin{pmatrix}1&0\\0&-1\end{pmatrix},\qquad
 M_\infty:=C\beta^{-\sigma_3}C^{-1}.
\end{equation*}
Here $\sigma_3$ is the third Pauli matrix. Thus
$z^{\sigma_3}=\diag(z,z^{-1})$ for $z\ne0$, where $\diag$ denotes
a diagonal matrix.
In the $2\times2$ matrix formulas below, $I_2$ is the identity
matrix, and $E_{12},E_{21}$ are the matrix units
\[
 E_{12}:=\begin{pmatrix}0&1\\0&0\end{pmatrix},\qquad
 E_{21}:=\begin{pmatrix}0&0\\1&0\end{pmatrix}.
\]
With the orientations of \cite{BBD}, the transformed jumps are
\begin{equation*}
 \begin{pmatrix}0&1\\-1&0\end{pmatrix}\quad\text{on the band},\qquad
 I_2+e^{-2k\alpha}E_{12}\quad\text{on the gap},\qquad
 I_2+e^{2k\alpha}E_{21}\quad\text{on the lips}.
\end{equation*}
The phase vanishes at the relevant endpoint; each lip uses the sheet
with negative real part.  Reflection $x\mapsto-x$ gives the other
endpoint, also for complex $b$.  We shall use
\begin{align}
 d\alpha_b&=2i\sqrt{1+b^2}\,
 \frac{x\sqrt{x^2-b^2}}{(1+x^2)^2}\,dx,\label{rh:phasedifferential}\\
 \alpha_b(x)&=i\left(\arctan v-\frac{v}{1+v^2}\right),\qquad
 v=\sqrt{\frac{x^2-b^2}{1+b^2}},\notag\\
 \partial_h\alpha_b(x)&=-\frac{i\sqrt{x^2-h}}{\sqrt{1+h}(1+x^2)},
 \qquad h=b^2.\label{rh:phaseparameter}
\end{align}
Here $\partial_h\alpha_b$ means that $\alpha_b$ is viewed as a function
of $h=b^2$ on the selected branch of $b$.
Phase periods are integer multiples of $\pi i$, independent of $h$,
so $e^{2k\alpha}$ is single-valued.  On the parameter disk,
\begin{equation*}
 \left|b^2/b_0^2-1\right|\leq1/5000,\qquad
 |b/b_0-1|<0.000101.
\end{equation*}

The opening angle satisfies
\begin{equation*}
 0.26<2\tan(\pi/24)<0.27.
\end{equation*}
The error contour consists of the images under $M$ of
$|x\mp b_0|=b_0/10$, the gap $[-0.9b_0,0.9b_0]$, and the four lips
\begin{equation*}
 x=b_0+re^{\pm i\pi/24},\qquad
 x=-b_0-re^{\pm i\pi/24},\qquad r\geq b_0/10,
\end{equation*}
truncated at $|x|=0.01$, together with four connectors and the
retained portions $|x|\geq0.02$ of
\begin{equation*}
 z_\pm(\phi)=\exp\{\pm2\tan(\pi/24)\cos(\phi/2)+i\phi\},
 \qquad -\pi\leq\phi\leq\pi.
\end{equation*}
In the first quadrant the connector is $x(R)=Re^{ia(R)}$,
$0.01\leq R\leq0.02$, where $a$ is affine and matches the two endpoint
arguments; reflect it into the other quadrants.  Denote the resulting
fixed contour by $\Sigma$. Write
$O_\pm:=M(\{x:|x\mp b_0|<b_0/10\})$ and orient both
$\partial O_\pm$ counterclockwise, as in \cite[between (59) and (60)]{BBD}.
The long band has identity error jump.
Inside each circle the band, gap, and lips are joined to the moving
endpoint and canceled by the local model below.

The following geometric data will also control the Cauchy operator.
On the exterior outer lip, writing $u=\cos(\phi/2)$ gives
\begin{equation*}
 \begin{aligned}
 \Re x&=\frac{\sin\phi}{\cosh(2u\tan(\pi/24))-\cos\phi},\\
 \Im x&=\frac{\sinh(2u\tan(\pi/24))}{\cosh(2u\tan(\pi/24))-\cos\phi},\\
 |x|^2&=\frac{\cosh(2u\tan(\pi/24))-1+2u^2}
                   {\cosh(2u\tan(\pi/24))+1-2u^2}.
 \end{aligned}
\end{equation*}
The straight and outer connector-end arguments lie respectively in
$[\pi/24-0.0002,\pi/24]$ and
$[\pi/24,\pi/24+0.0003]$.  Hence
\begin{equation}\label{rh:connectorbounds}
 |a-\pi/24|<0.001,\quad |a'|<0.2,\quad
 |x'|\leq1.01,\quad |x''|\leq1,
 \qquad |x(R)-x(S)|\geq0.99|R-S|.
\end{equation}
The identity
\begin{equation}\label{rh:cayleydifference}
 M(y)-M(x)=\frac{-2i(y-x)}{(y-i)(x-i)}
\end{equation}
then gives connector $z$-speeds in $[1.8,3]$, chord ratio at least
$1.8$, second derivative at most $10$, and parameter length $0.01$.

For clarity, the connector endpoint bounds can be checked directly.
At the straight endpoint $|x|=0.01$, the angular difference from
$\pi/24$ is at most
$b_0\sin(\pi/24)/(0.01-b_0)<0.000014$.
At the outer endpoint $|x|=0.02$, the Cayley expression implies
$u\le0.02$ and
\[
 \tan(\arg x)=\tan(\pi/24)
 \frac{\sinh(\epsilon u)}{\epsilon u\sqrt{1-u^2}},
 \qquad \epsilon=2\tan(\pi/24).
\]
The last ratio is below $1.00021$, so the angular difference is below
$0.00003$. Affine interpolation gives $|a'|<0.05$, stronger than
\eqref{rh:connectorbounds}. On a connector
$2/(1.02)^2\le|M'|\le2/(0.98)^2$, which proves its stated speed,
chord, and second-derivative bounds.
On an outer lip $|x|<20$: for $u\le0.9$, the denominator in the
formula for $|x|^2$ is at least $0.38$; for $u\ge0.9$, use
$|x|\le\coth(0.9\epsilon/2)<20$.
At a retained endpoint $u>0.01$, since otherwise that same formula
would give $|x|<0.0101$, rather than $0.02$.
Thus the angular endpoints are more than $0.02$ from $\pm\pi$, and
the two outer radii differ by at least $2\sinh(0.01\epsilon)>0.005$.
These direct inequalities justify the geometric data used in the
Cauchy-operator estimate below.

\begin{lemma}\label{rh:phasemargins}
On a straight lip, with $R=r/b_0\geq1/10$,
\begin{equation}\label{rh:lipphase}
 \Re\alpha_b(x)\leq-\frac{b_0^3}{20}R^{3/2}(1+R)^{3/2}.
\end{equation}
On the gap, $\Re\alpha_b\geq b_0^3/40$; on a connector,
$\Re\alpha_b\leq-10^{-7}$; on a retained outer lip,
$\Re\alpha_b\leq-10^{-6}/64$.
\end{lemma}
\begin{proof}
Write $a=\pi/24$ and $h=b^2$. On the first lip put
$x=b_0+re^{ia}=b(1+qe^{i\eta})$ and
$\delta=b/b_0-1$. Then
\[
 \frac{(x-b)/b}{Re^{ia}}=
 \frac{1-\delta e^{-ia}/R}{1+\delta},\qquad
 \left|\frac{(x-b)/b}{Re^{ia}}-1\right|<0.001112.
\]
Hence $0.997<q/R<1.003$ and $|\eta-a|<0.002$.
Along $y=b(1+ve^{i\eta})$, $0\le v\le q$, the phase differential is
\[
 \frac{2i\sqrt{1+h}\,b^3e^{i\eta}(1+ve^{i\eta})
             \sqrt{ve^{i\eta}}\sqrt{2+ve^{i\eta}}}{(1+y^2)^2}\,dv.
\]
The $v$-dependent numerator, after removing $i$, has argument between
$3\eta/2$ and $3\eta$. The remaining argument is bounded in modulus by
\[
 \tfrac12\arcsin(1.0002\cdot10^{-12})
 +3\arcsin(0.000101)+2\arcsin(10^{-4})<0.001.
\]
Its total argument lies between $3a/2-0.004$ and $3a+0.007$;
its sine exceeds $0.19$, by $\sin u\ge u-u^3/6$.
Using $\eta<0.133$ and
$|j+ve^{i\eta}|\ge\cos\eta(j+v)$ for $j=1,2$, its modulus is at least
\[
 \frac{2\sqrt{1-1.0002\cdot10^{-12}}
       (1-0.000101)^3\cos(0.133)^{3/2}}{(1.0001)^2}
 b_0^3(1+v)\sqrt{v(2+v)}\,dv
 >1.9b_0^3(1+v)\sqrt{v(2+v)}\,dv.
\]
Integration uses
$\int_0^q(1+v)\sqrt{v(2+v)}\,dv=[q(2+q)]^{3/2}/3$.
Since $q(2+q)\ge0.994009R(1+R)$, the resulting coefficient is
$(0.19\cdot1.9/3)(0.994009)^{3/2}>0.11>1/20$.
This proves \eqref{rh:lipphase} for complex $b$ throughout the disk.
The other lips follow by simultaneous spatial and parameter conjugation,
and by $x\mapsto-x$ with the endpoint-normalized reflected phase sheet.

On the reference gap, $x=b_0v$, $|v|\le0.9$, one has
\[
 \alpha_{b_0}(b_0v)
 =2\sqrt{1+b_0^2}\,b_0^3
   \int_{|v|}^1\frac{u\sqrt{1-u^2}}{(1+b_0^2u^2)^2}\,du
 \ge\frac{2(1-v^2)^{3/2}b_0^3}{3(1+b_0^2)^{3/2}}
 >0.054b_0^3.
\]
For example $0.19^3>0.082^2$ verifies the last numerical margin.
Along the straight parameter segment, \eqref{rh:phaseparameter} has
modulus below $1.0002b_0$, since
$|x^2-h|\le(1-v^2+0.0002)b_0^2$. Its integrated perturbation is
less than $0.000201b_0^3<0.00028b_0^3$, leaving $b_0^3/40$.
The segment avoids the phase branch point because $x^2\le0.81b_0^2$.

For the outer pieces put $\epsilon=2\tan a$, $u=\cos(\phi/2)$ and
$v=\epsilon u$. The critical phase, on the appropriate sheet, satisfies
\[
 \Re\alpha_0=-\tfrac12(v+\sinh v\cos\phi).
\]
Since $\cos\phi=2u^2-1$, $\epsilon\in(0.26,0.27)$, and
$\sinh v-v\le v^3\cosh v/6$,
\[
 v+\sinh v\cos\phi
 \ge[2\epsilon-\epsilon^3\cosh(\epsilon)/6]u^3
 >\tfrac12u^3.
\]
The Cayley formula above shows that $u\le0.005$ would imply
$|x|<0.0051$: its numerator is below $2.038u^2$ and denominator
at least $2-2u^2$. Thus $|x|\ge0.01$ implies $u>0.005$, and
$\Re\alpha_0<-10^{-6}/32$.
For $|h|\le2\cdot10^{-12}$,
\[
 |\partial_h\alpha_b(x)|
 \le\frac{|x|}{|1+x^2|}
       \frac{\sqrt{1+|h|/|x|^2}}{\sqrt{1-|h|}}<1.
\]
Indeed
$|x|/|1+x^2|=|z-z^{-1}|/4
=\tfrac12\sqrt{\sinh^2v+\sin^2\phi}<0.519$,
and the remaining factor is below $1.00000002$.
Perturbing from $h=0$ therefore costs less than $2b_0^2$, leaving
$\Re\alpha_b<-10^{-6}/64$.

Finally, on a first-quadrant connector
$0.01\le|x|\le0.02$ and $\pi/48\le\arg x\le\pi/12$.
The exact critical-phase series gives
\[
 \left|\alpha_0(x)-\frac{2i}{3}x^3\right|
 \le\frac{|x|^5}{1-|x|^2}.
\]
Since $\sin(3\arg x)\ge\sin(\pi/16)>0.195$,
$\Re\alpha_0<-0.12|x|^3$.
On the parameter segment from zero to $h$,
\[
 |\partial_h\alpha_b(x)|\le
 \frac{\sqrt{0.02^2+2\cdot10^{-12}}}
 {\sqrt{1-2\cdot10^{-12}}(1-0.02^2)}<0.021.
\]
The perturbation is less than $10^{-13}$, so
$\Re\alpha_b\le-1.2\cdot10^{-7}+10^{-13}<-10^{-7}$.
Here the parameter segment avoids $h=x^2$ since $|x|\ge0.01$.
The reflected connectors follow as for the straight lips.
All parameter derivatives are of the specified phase sheets; their
periods are constant integer multiples of $\pi i$, so no unknown
parameter-dependent additive constant is introduced.
\end{proof}

Put $\omega:=e^{2\pi i/3}$ and
$\mathbf a_j(s):=(\Ai(\omega^js),\omega^j\Ai'(\omega^js))^T$.
The Airy matrix is
\begin{equation}\label{rh:airysectors}
 \Psi(s):=\begin{cases}
 [\mathbf a_0,\mathbf a_2]D,&0<\arg s<2\pi/3,\\
 [-\omega\mathbf a_1,\mathbf a_2]D,&2\pi/3<\arg s<\pi,\\
 [-\omega^2\mathbf a_2,-\omega^2\mathbf a_1]D,&-\pi<\arg s<-2\pi/3,\\
 [\mathbf a_0,-\omega^2\mathbf a_1]D,&-2\pi/3<\arg s<0,
 \end{cases}\qquad
 D:=\diag(e^{-i\pi/6},e^{i\pi/6}).
\end{equation}
The connection and Wronskian identities in
\cite[(9.2.8), (9.2.12)]{DLMF} give
\begin{equation}\label{rh:airywronskian}
 \det\Psi(s)=e^{i\pi/6}/(2\pi).
\end{equation}
Each displayed sector formula is entire before the sector selection.

\begin{lemma}\label{rh:airybound}
For $s\ne0$, using consistent sector logarithms, set $\zeta:=2s^{3/2}/3$ and
\[
 N(s):=\frac{e^{-\zeta\sigma_3}\Psi(s)^{-1}s^{-\sigma_3/4}C^{-1}}
              {\sqrt\pi e^{i\pi/6}},\qquad
 A_1:=\begin{pmatrix}-1/72&-i/12\\-i/12&1/72\end{pmatrix}.
\]
Then $\|A_1\|_{\mathrm F}<1/8$ and
$\|N-I_2-A_1/\zeta\|_{\mathrm F}\leq2/(5|\zeta|^2)$ in the standard
sectors, including boundary values.  For the same column formulas
continued while every scalar Airy argument satisfies
$|\arg y|\leq\pi-0.1$,
\begin{equation*}
 \|N-I_2-A_1/\zeta\|_{\mathrm F}\leq1000/|\zeta|^2.
\end{equation*}
\end{lemma}
\begin{proof}
The two-term expansions of $\Ai(y)$ and $\Ai'(y)$ have normalized
coefficients $5/72$ and $-7/72$, and first omitted coefficients of
modulus $385/10368$ and $455/10368$.  Apply
\cite[(9.7.17)]{DLMF} with $n=2$, using $\sigma=1/6$ and $0$,
respectively.  Its angular multiplier is below $4$ for
$|\arg y|\leq2\pi/3$, and below $4096$ for
$|\arg y|\leq\pi-0.1$: here
$|\cos(3\arg y/2)|\geq\sin(0.15)>1/8$ in the additional angular
range and $\chi(13/6)<3$.  Thus the two normalized scalar remainders
satisfy
\begin{equation*}
 |r_c|,|r_d|\leq\frac1{5|\zeta|^2}
 \quad\text{in the standard range},\qquad
 |r_c|,|r_d|\leq\frac{180}{|\zeta|^2}
 \quad\text{in the continued range}.
\end{equation*}
The scalar argument pairs in the four sectors, with $\theta=\arg s$,
are $(\theta,\theta-2\pi/3)$, $(\theta-4\pi/3,\theta-2\pi/3)$,
$(\theta+4\pi/3,\theta+2\pi/3)$, and
$(\theta,\theta+2\pi/3)$.  Their exponents are exactly $\zeta$ and
$-\zeta$.  Adjugate inversion using \eqref{rh:airywronskian} gives
$N=I_2+A_1/\zeta+EC^{-1}$, where the four entries of $E$ are scalar
remainders multiplied by unit complex numbers and
$\|EC^{-1}\|_{\mathrm F}=\|E\|_{\mathrm F}/\sqrt2$.
This proves both estimates, including the continued Stokes terms.
\end{proof}

\begin{lemma}\label{rh:fixedmatching}
The local parametrices can be chosen with fixed band, gap, and lip
exits.  On each fixed boundary arc their error jump $I_2+w$ is
holomorphic in $t$, and its first matching coefficient $V_1$ satisfies
\begin{equation}\label{rh:matchingconstants}
 |k\alpha|\geq\Lambda_0/100,\qquad
 \|V_1\|_{\mathrm F}\leq50/\Lambda_0,\qquad
 \|w-V_1\|_{\mathrm F}\leq5\cdot10^7/\Lambda_0^2.
\end{equation}
The two circles have total $z$-length less than $3b_0$.
\end{lemma}
\begin{proof}
Near $x=b$, with $w_b:=x/b-1$,
\begin{align}
 \alpha&=\frac{4i\sqrt2}{3}\sqrt{1+b^2}\,b^3w_b^{3/2}H(w_b,b^2),
 \label{rh:localphase}\\
 H(w,v)&:=\frac32\int_0^1\tau^{1/2}
 \frac{(1+\tau w)\sqrt{1+\tau w/2}}
      {[1+v(1+\tau w)^2]^2}\,d\tau.\notag
\end{align}
For $|w|\leq0.102$, $|v|\leq1.0002\cdot10^{-12}$, the integrand
bounds give
\[
 |H-1|<0.132,\qquad |H_w|<1,\qquad |H_v|<2.8.
\]
These bounds follow directly with the normalized integration weight
$(3/2)\tau^{1/2}d\tau$. In fact, for
$A_\tau=(1+\tau w)\sqrt{1+\tau w/2}$ and
$D_\tau=[1+v(1+\tau w)^2]^{-2}$, put
$r_v=|v|(1.102)^2<1.216\cdot10^{-12}$. Then
\[
 |A_\tau-1|\le0.102+
   \frac{1.102\cdot0.051}{1+\sqrt{0.949}}<0.1305,\qquad
 |D_\tau-1|\le\frac{2r_v+r_v^2}{(1-r_v)^2}<2.44\cdot10^{-12}.
\]
Also $|\partial_w(A_\tau D_\tau)|<1.31\tau$, and
$|H_v|\le2(1.102)^3\sqrt{1.051}/(1-r_v)^3<2.8$.
Integration gives the displayed estimates.

Use the coordinate
$\lambda:=-2b^2w_b(1+b^2)^{1/3}H^{2/3}$ and $s:=k^{2/3}\lambda$;
the branches are specified below. For
$G=(1+b^2)^{1/3}H^{2/3}$, the segment formula for fractional powers
and $|H|\ge0.868$ give
\[
 |G-1|<0.093,\qquad |wG_w|<0.072,
 \qquad |\partial_w(wG)-1|<0.165<0.182.
\]
The last bound proves conformality and injectivity on the convex
local $w$-domain by integration along a chord. On the fixed circle,
$w_0=x/b_0-1$ satisfies $|w_0|=0.1$, and
\[
 |w_b-w_0|<0.000112,\qquad0.0998<|w_b|<0.101,
 \qquad
 \left|\frac{H(w_b,b^2)}{H(w_0,b_0^2)}-1\right|<0.0001291.
\]
Together with $|b^2/b_0^2-1|\le0.0002$ and
$|w_b/w_0-1|<0.00112$, this gives
\[
 |\lambda_b/\lambda_{b_0}-1|<0.00141<0.002,
 \qquad |\arg(\lambda_b/\lambda_{b_0})|<0.003.
\]
At the real reference parameter,
\[
 \frac{d\arg\lambda}{d\arg w}
 =\Re\left(1+\frac{2wH_w}{3H}\right)>0.92.
\]
At $w=0.1e^{\pm i\pi/24}$, comparison with the positive real
$H(\Re w,b_0^2)$ gives
\[
 |\arg H|\le\arcsin(0.1\sin(\pi/24)/0.868)<0.016.
\]
These estimates supply the angular margins used in the fixed
sector construction below.

Here is a direct verification of the exact matching, including the
nonstandard lip exits.  It uses the Airy connection identity and the
real transformed jump problem of Baik--Buckingham--DiFranco~\cite{BBD} (BBD), not a complex-parameter theorem
from that paper.
Put $a:=\pi/24$, $\mathbb B:=\{t:|t-t_0|<\Delta_0/10^4\}$, and
$U_+:=\{x:|x-b_0|<b_0/10\}$.  The four fixed exits, in counterclockwise
order, are
\[
 e_0:=b_0(1+1/10),\quad
 e_1:=b_0(1+e^{ia}/10),\quad
 e_2:=b_0(1-1/10),\quad
 e_3:=b_0(1+e^{-ia}/10).
\]
The first is the band exit, the second and fourth are lip exits, and
the third is the gap exit.  Join $b=b(t)$ to these exits by straight
segments.  Since $|b-b_0|<0.000101b_0$, they are distinct and retain
this cyclic order.  Orient the band and both lips toward $b$, and
the gap away from $b$.  These are the BBD orientations at the upper
endpoint.  The Cayley map is conformal here and preserves the
indicated sides.

Write $w=x/b-1$, $\gamma=1+b^2$, and
\[
 c_b:=\frac{-b-i}{b-i},\qquad
 \varphi:=\log 2+2\log b+\frac13\log\gamma+\frac23\log H(w,b^2).
\]
The logs of $b,\gamma$ are selected by their positive real values at
$t_0$, the log of $H$ by $|H-1|<0.132$, and the log of $c_b$ near zero.
The function $\log(2+w)$ is the branch near $\log 2$.
These are jointly holomorphic on the full local disk and parameter
disk.  Define the \emph{single-valued} Airy coordinate
\begin{equation*}
 \lambda=-e^{\varphi}w
 =-2b^2\gamma^{1/3}wH(w,b^2)^{2/3},\qquad s=k^{2/3}\lambda.
\end{equation*}

Cut $U_+$ along the segment from $b$ to $e_0$.  On its complement
choose $\ell=\log w$ so that, at the reference real parameter,
$0<\operatorname{Im}\ell<2\pi$.  This specifies its two boundary
values, not a principal-log convention after parameter perturbation.
For a fixed boundary point $x=b_0(1+0.1e^{i\theta})$, $0<\theta<2\pi$,
it is explicitly
\begin{equation*}
 \ell=\log(0.1)+i\theta+
 \operatorname{Log}\!\left(\frac{x/b-1}{0.1e^{i\theta}}\right),
\end{equation*}
where the last Log is the branch near zero.  Its input differs
from $1$ by less than $0.00112$, so this formula is holomorphic on
$\mathbb B$.  Continue it spatially through the slit disk.  On each
segment $x=b+v(e_r-b)$, $0<v\le1$, it equals
$\log v+\ell(e_r)$; consequently the spatial and boundary choices
agree.  In particular the cut and logs are specified on the whole
parameter disk, rather than assigned by a changing principal
argument.

Set
\begin{equation}\label{patch:branchlogs}
 \log\lambda=\ell+\varphi-i\pi,\qquad
 \log\beta=\frac14\{\ell-\log(2+w)+\log c_b\}.
\end{equation}
The second choice is the restriction of the outer fourth root:
its fourth power is exactly $(z-\xi_+)/(z-\xi_-)$, it has the
correct band jump, and at the real reference parameter it agrees
with the fourth root normalized at infinity.  The ratio of these
two roots is a fourth root of unity, locally constant in both
variables, hence remains one throughout the connected parameter
disk.  The outside portion of the band is fixed; only its part
inside $U_+$ is the specified segment.

With all powers interpreted using \eqref{patch:branchlogs},
\[
 \zeta=\frac23s^{3/2}
 =\frac{4i\sqrt2}{3}kb^3\gamma^{1/2}e^{3\ell/2}H(w,b^2)
 =k\alpha.
\]
This is an algebraic equality of the selected branches.  In
particular, at a real parameter the gap has $\alpha>0$.  Crossing
the band from the upper $x$ side to the lower $x$ side increases
$\ell$ by $2\pi i$, so $\beta_+=i\beta_-$ and
$\zeta_+=-\zeta_-$.

Let $S_1,S_2,S_3,S_4$ be the four sectors encountered
counterclockwise between the segments to
$(e_0,e_1),(e_1,e_2),(e_2,e_3),(e_3,e_0)$, respectively.
They are spatial sectors with these fixed exit labels; they are
not reselected by the principal argument of $s$.
Using $D$ and $\mathbf a_j$ from~\eqref{rh:airysectors}, put
\[
 \begin{array}{c|c}
 \text{physical sector}&\Psi_*\\ \hline
 S_1&[-\omega^2\mathbf a_2,-\omega^2\mathbf a_1]D\\
 S_2&[\mathbf a_0,-\omega^2\mathbf a_1]D\\
 S_3&[\mathbf a_0,\mathbf a_2]D\\
 S_4&[-\omega\mathbf a_1,\mathbf a_2]D.
 \end{array}
\]
These are entire column formulas.  The Airy identity
$\mathbf a_0+\omega\mathbf a_1+\omega^2\mathbf a_2=0$ gives,
by direct column multiplication,
\begin{equation}\label{patch:connections}
 \begin{gathered}
 \Psi_{S_1}=\Psi_{S_2}(I_2+E_{21}),\qquad
 \Psi_{S_3}=\Psi_{S_4}(I_2+E_{21}),\\
 \Psi_{S_3}=\Psi_{S_2}(I_2+E_{12}),\qquad
 \Psi_{S_4}=\Psi_{S_1}J_0,\qquad
 J_0:=\begin{pmatrix}0&1\\-1&0\end{pmatrix}.
 \end{gathered}
\end{equation}
The last identity compares the two sides of the band at the same
value of $s$; the chosen logarithms of $s$ differ by $2\pi i$ there.

To make removal of the fourth-root singularity explicit, define
\[
 \mathcal A:=\exp\left\{\frac23\log k+\varphi+
                    \log(2+w)-\log c_b-i\pi\right\}
            =\frac{s}{\beta^4},
 \qquad c_0:=\sqrt\pi e^{i\pi/6},
\]
and set
\begin{equation}\label{patch:localP}
 P^{(+)}(z;t):=c_0 C\mathcal A^{\sigma_3/4}
                 \Psi_*(s)e^{\zeta\sigma_3}.
\end{equation}
The displayed logarithm of $\mathcal A$ is holomorphic on the
\emph{full} local disk and on $\mathbb B$.  Thus the prefactor in
\eqref{patch:localP} is holomorphic and invertible at $x=b$.
The entire Airy formulas and $\zeta\to0$ make $P^{(+)}$ bounded and
invertible sectorwise at the endpoint.  Its determinant is one,
using $\det\Psi=e^{i\pi/6}/(2\pi)$ and $\det C=-2i$.

The superscripts $(+)$ and $(-)$ label the upper and lower endpoint
models; subscripts $+$ and $-$ denote boundary traces.
Here are all jump computations, with the indicated orientations.
At the upper lip, the plus/minus sectors are $S_1/S_2$; at the
lower lip they are $S_3/S_4$; at the gap they are $S_3/S_2$;
and at the band they are $S_4/S_1$.  Therefore
\eqref{patch:connections} and
\[
 e^{-\zeta\sigma_3}E_{21}e^{\zeta\sigma_3}
     =e^{2\zeta}E_{21},\qquad
 e^{-\zeta\sigma_3}E_{12}e^{\zeta\sigma_3}
     =e^{-2\zeta}E_{12},\qquad
 J_0\sigma_3=-\sigma_3J_0
\]
give, respectively,
\[
 \begin{gathered}
 (P^{(+)})_+=(P^{(+)})_-(I_2+e^{2k\alpha}E_{21}),\qquad
 (P^{(+)})_+=(P^{(+)})_-(I_2+e^{2k\alpha}E_{21}),\\
 (P^{(+)})_+=(P^{(+)})_-(I_2+e^{-2k\alpha}E_{12}),\qquad
 (P^{(+)})_+=(P^{(+)})_-J_0.
 \end{gathered}
\]
All these identities hold pointwise for every complex parameter in
$\mathbb B$.  They do not depend on the segments being inverse
images of straight Airy rays.

On each fixed circle arc the foregoing logarithms give
$\mathcal A^{\sigma_3/4}=s^{\sigma_3/4}\beta^{-\sigma_3}$.
Consequently, with $M_\infty=C\beta^{-\sigma_3}C^{-1}$ and with
$N$ defined using the assigned entire formula $\Psi_*$,
\[
 P^{(+)}=M_\infty N^{-1},\qquad
 N=\frac{e^{-\zeta\sigma_3}\Psi_*^{-1}
                   s^{-\sigma_3/4}C^{-1}}{c_0}.
\]
The circle is counterclockwise, so its plus side is interior.
The exact error jump is therefore
\[
 v_R=M_\infty (P^{(+)})^{-1}=M_\infty N M_\infty^{-1}.
\]
This proves the asserted matching identity, with parameter
holomorphy on every fixed open circle arc.

The scalar Airy arguments needed for the remainder estimate follow
from the same fixed assignment.  At the real reference parameter,
$\arg\lambda=\arg w-\pi+(2/3)\arg H$; the scalar argument pairs on
$S_1,S_2,S_3,S_4$ are, respectively,
\[
 \begin{gathered}
 (\arg\lambda+4\pi/3,\arg\lambda+2\pi/3),\quad
 (\arg\lambda,\arg\lambda+2\pi/3),\quad 
 (\arg\lambda,\arg\lambda-2\pi/3),\quad
 (\arg\lambda-4\pi/3,\arg\lambda-2\pi/3).
 \end{gathered}
\]
Use the continuous log of $\lambda$ just specified in these formulas.
The monotonicity and perturbation estimates already proved give a
margin of at least
$a-2(0.016)/3-0.003>0.1$ for the only scalar arguments approaching
the negative real axis.  The other scalar arguments remain much
farther away.  Thus all four fixed assignments satisfy
$|\arg y|\le\pi-0.1$ uniformly on $\mathbb B$, including their
one-sided boundary values.  No discontinuous principal-argument
reselection occurs when the parameter is nonreal.

At the other endpoint define, explicitly,
\[
 P^{(-)}(z;t)=\overline{P^{(+)}(\bar z;\bar t)}.
\]
Simultaneous conjugation of $z$ and $t$ preserves holomorphy.
Reflecting the upper contours and reversing their orientations
preserves the jump-side convention. The reflected phase vanishes
at the lower endpoint, and
$M_\infty(z;t)=\overline{M_\infty(\bar z;\bar t)}$
by the reciprocal transformation of its fourth root.
Hence the lower construction satisfies the reflected jump and
matching conditions, with the same norm bounds.

For real $t\in\mathbb B$, the band and gap remain on the unit
circle. Each lip can be deformed to the chosen segment without
crossing either; the matrix $I_2+e^{2k\alpha}E_{21}$ and its inverse
are analytic in the swept lens and transfer the old jump to the
new contour. Since this deformation leaves neighborhoods of
$0$ and $\infty$ unchanged, the original recovery identities
remain valid for real $t$. Together with the complex jump matching
and fixed-boundary parameter holomorphy established above, this
justifies their later meromorphic continuation for the recovered
determinant quotients.
 
 Equation \eqref{rh:localphase} gives $|k\alpha|\geq\Lambda_0/40$.
Moreover,
\[
 \beta^4=\frac{w_b}{2+w_b}\frac{-b-i}{b-i},\qquad
 \max(|\beta|^2,|\beta|^{-2})<5.
\]
Since $C/\sqrt2$ is unitary, conjugation by $M_\infty$ has Frobenius
operator norm at most $5$ on the circles.  The exact matching is
\[
 I_2+w=M_\infty N M_\infty^{-1},\qquad
 V_1=\frac{M_\infty A_1M_\infty^{-1}}{k\alpha}.
\]
Lemma~\ref{rh:airybound} now gives \eqref{rh:matchingconstants}.
The length bound follows from $|M'|<2.01$ on the two circles.
Fixed entire column formulas, holomorphic prefactors, and the
nonzero Wronskian give parameter holomorphy on each boundary arc.
\end{proof}

For $f\in L^2(\Sigma)$, define the Cauchy transform and its
negative-side boundary operator by
\begin{equation*}
 (\mathcal C_-f)(t):=-\tfrac12f(t)+\frac1{2\pi i}
           \operatorname{p.v.}\int_\Sigma\frac{f(s)}{s-t}\,ds,
 \qquad
 (\mathcal Cf)(z):=\frac1{2\pi i}\int_\Sigma\frac{f(s)}{s-z}\,ds.
\end{equation*}
Here $\operatorname{p.v.}$ denotes the Cauchy principal value.
Traces are taken on open smooth arcs; the finitely many vertices
have measure zero.  The $L^2$ Cauchy trace and Plemelj theorems
\cite[Theorems~2.2--2.3]{Lenells} apply locally by completing a
smooth subarc to a piecewise smooth Jordan curve and extending its
density by zero.  The remaining integral is analytic near a compact
subarc.  This gives $\mathcal C_+-\mathcal C_-=I$ without imposing a
global two-coloring of the regions of the error network.
The notation $\|\mathcal C_-\|_{2\to2}$ denotes the induced operator
norm from $L^2(\Sigma)$ to $L^2(\Sigma)$, both with arclength measure:
\[
 \|\mathcal C_-\|_{2\to2}:=
 \sup_{0\ne f\in L^2(\Sigma)}
 \frac{\|\mathcal C_-f\|_{L^2(\Sigma)}}{\|f\|_{L^2(\Sigma)}}.
\]
Thus both the input and output norms are $L^2$ function norms.
This operator norm is distinct from the Hilbert--Schmidt norm
$\|\cdot\|_2$ of an operator defined above. We use the same $2\to2$
notation for operators between the $L^2$ spaces of individual contour
pieces. We now bound the Cauchy operator on this contour explicitly.

For a twice continuously differentiable chart
$\varphi:[0,\ell]\to\mathbb C$, choose constants $0<m\leq M_*$ and
$B_*\geq0$ such that $m\leq|\varphi'(u)|\leq M_*$,
$|\varphi''(u)|\leq B_*$, and
$|\varphi(u)-\varphi(v)|\geq m|u-v|$ for all $u,v\in[0,\ell]$.
These bounds may vary with the chart and need not be optimal.
Comparison with the straight Cauchy projection and the Schur test gives
the self-piece bound
\begin{equation}\label{rh:selfchart}
 \|\mathcal C_-\|_{2\to2}
 \leq\sqrt{M_*/m}\left(1+\frac{B_*\ell}{4\pi m}\right).
\end{equation}
Indeed the difference of the pulled-back kernel and $1/(u-v)$ is
bounded by $B_*/(2m)$.  The straight-line and circular Cauchy
projections have norm $1$.  For incident charts parametrized away
from their common vertex, speeds at most $M_1,M_2$ and separation
$c(s+u)$ give
\begin{equation}\label{rh:carleman}
 \|\mathcal C_{12}\|_{2\to2}\leq\frac{\sqrt{M_1M_2}}{2c};
\end{equation}
the weighted Schur test uses
$\int_0^\infty u^{-1/2}(s+u)^{-1}du=\pi s^{-1/2}$.
For pieces of lengths $\ell_1,\ell_2$ separated by $d_*$,
\begin{equation}\label{rh:hilbertschmidt}
 \|\mathcal C_{12}\|_{2\to2}\leq
 \frac{\sqrt{\ell_1\ell_2}}{2\pi d_*}.
\end{equation}
These elementary estimates also apply in the reverse direction.

\begin{lemma} \label{rh:cauchy}
The fixed contour satisfies
\begin{equation}\label{rh:cauchybound}
 \|\mathcal C_-\|_{L^2(\Sigma)\to L^2(\Sigma)}\leq10^{10},
 \qquad \inf_{s\in\Sigma}|s|\geq\tfrac12.
\end{equation}
The same operator bound holds for Frobenius-valued $L^2$ functions.
\end{lemma}
\begin{proof}
Group the two circles, four straight lips, and gap together as the
local part.  In the $x$ coordinate their self-piece norms are at most
$1$ and their cross-piece norms at most $4$.  For lips meeting the
same endpoint the separation is at least $\sin(\pi/24)(r+s)$;
opposite lips and gap--lip pairs have a larger real separation.
For a radius-$a$ circle and an incident radial segment, either
half-circle, parametrized from the vertex, satisfies
\[
 |ae^{iu/a}-(a+s)|\geq(s+u)/3.
\]
Thus its two contributions have combined norm below $3$.
Opposite circle--ray pairs have separation at least $0.9(b_0+s)$
and Hilbert--Schmidt norm below $1$; the two circles have separation
at least $1.8b_0$ and the same bound.  The local block matrix has row
and column sums at most $25$.

The local $x$-length is less than $0.05$, and
\begin{equation*}
 \frac{M'(y)}{M(y)-M(x)}=\frac1{y-x}-\frac1{y-i}.
\end{equation*}
The extra kernel has norm below $0.01$, while the square-root
arclength-weight ratio is below $1.021$.  Hence the local operator
in $z$-arclength has norm below $26$.

For a connector, the data in \eqref{rh:connectorbounds} give
$(m,M_*,B_*,\ell)=(1.8,3,10,0.01)$ in
\eqref{rh:selfchart}, hence a bound below $4$.
Each retained outer lip is a single $\phi$-chart of length below $7$,
with endpoints at least $0.02$ from $\pm\pi$, speeds in $[0.7,2]$,
and second derivative at most $3$.  Its chord bound is
\[
 |z_\pm(\phi)-z_\pm(\psi)|\geq
 1.4\sin(|\phi-\psi|/2)\geq0.001|\phi-\psi|.
\]
Thus \eqref{rh:selfchart}, with $(m,M_*,B_*,\ell)=(0.001,2,3,7)$,
is below $10^5$.

We regard the contour as seven groups: the local part, four
connectors, and two outer lips.  The cross-group bounds are as follows.
\begin{enumerate}
\item A connector and its incident straight lip have radial
separation at least $0.99$ times the sum of their distances from the
vertex.  Equations \eqref{rh:cayleydifference} and
\eqref{rh:carleman} bound this operator by $100$.
\item For a connector and its incident outer lip, let
$s=0.02-R$ and let $u$ be angular distance from their common
endpoint.  The outer chord bound gives distance at least
$0.001u-3s$.  On the outer lips $|x|<20$, so
\eqref{rh:cayleydifference} also gives distance at least $s/12$.
Splitting at $u=6000s$ yields separation at least
$(s+u)/100000$.  Speeds at most $3$ and $2$ in
\eqref{rh:carleman} give norm below $10^6$.
\item All remaining connector--local and connector--connector
pairs have $z$-distance at least $10^{-4}$.  In the $x$ coordinate,
nonmatching same-side lips have imaginary separation greater than
$0.0012$, opposite-side lips have real separation greater than
$0.0099$, and the circles and gap have separation at least
$0.01-1.1b_0$.  The same sign separations cover distinct connectors.
\item An outer lip and the local part have $x$-radial separation
at least $0.01$, hence $z$-distance greater than
$0.02/(21\cdot1.01)>10^{-4}$.  The two outer lips have radial
$z$-separation greater than $0.005$.  A connector and the other
outer lip have radial $z$-separation greater than $0.002$, since
$|\Im x|>0.0012$ and $|x|\leq0.02$ on the connector.
\end{enumerate}
Every group has length below $20$, so each separated block in
\eqref{rh:hilbertschmidt} is below $10^6$.
For a connector against the entire local part, add its incident
contribution, below $100$, to a separated contribution bounded by
$\sqrt{0.11\cdot0.03}/(2\pi\cdot10^{-4})<100$.
Thus the seven-by-seven block norm matrix has row and column sums
below $10^5+6\cdot10^6<10^7$.  The block Schur test proves
\eqref{rh:cauchybound}, initially on smooth functions and then by
$L^2$ density.  Finally, local and connector points have $|z|>0.9$,
and outer points have $|z|\geq e^{-2\tan(\pi/24)}>0.7$.
The matrix-valued assertion follows by summing the four scalar
squared-norm estimates.
\end{proof}

\begin{lemma} \label{rh:jumpbounds}
Let $I_2+w$ be the error jump on $\Sigma$: on the circles it is the jump
in Lemma~\ref{rh:fixedmatching}; on the remaining arcs it is
$M_\infty J M_\infty^{-1}$, where $J$ is the corresponding transformed
gap or lip jump displayed above. Extend $V_1$ by zero off the two circles. Then
\begin{equation}\label{rh:jumpnorms}
 \|w\|_\infty\leq\frac{10^6}{\Lambda_0},\qquad
 \|w\|_2\leq\frac{10^7\sqrt{b_0}}{\Lambda_0},\qquad
 \|w-V_1\|_1\leq\frac{10^9b_0}{\Lambda_0^2}.
\end{equation}
The map $t\mapsto w(\cdot;t)$ is holomorphic with values in
$L^\infty(\Sigma)\cap L^2(\Sigma)$ on a neighborhood of the closed
parameter disk.
\end{lemma}
\begin{proof}
Outside the circles each endpoint denominator is at least
$0.0998b_0$, giving
$\|M_\infty A M_\infty^{-1}\|_{\mathrm F}\leq10\|A\|_{\mathrm F}$.
On the circles, Lemma~\ref{rh:fixedmatching} gives
$\|w\|_\infty\leq51/\Lambda_0$, an $L^2$ bound
$150\sqrt{b_0}/\Lambda_0$, and an $L^1$ remainder bound
$1.5\cdot10^8b_0/\Lambda_0^2$.
For the four straight lips, $|dz|\leq3b_0\,dR$ and
Lemma~\ref{rh:phasemargins} give
\[
 \int_{\rm lips}\|w\|_{\mathrm F}|dz|
 \leq120b_0\int_{1/10}^\infty e^{-\Lambda_0R^{3/2}/10}\,dR
 \leq\frac{1.2\cdot10^6b_0}{\Lambda_0^2}.
\]
The gap contributes at most $32000b_0/\Lambda_0^2$.
The combined outer and connector length is less than $21$ and their
exponential factor is at most $e^{-10^{-9}k}$; hence they contribute
at most
\[
 210e^{-10^{-9}k}\leq
 \frac{4.2\cdot10^{20}b_0^6}{\Lambda_0^2}
 \leq\frac{4.2\cdot10^{-10}b_0}{\Lambda_0^2}.
\]
Consequently, off the circles,
$\|w\|_1\leq10^7b_0/\Lambda_0^2$ and
$\|w\|_\infty\leq10^5/\Lambda_0$.
Interpolation $\|w\|_2^2\leq\|w\|_\infty\|w\|_1$ proves the
remaining bounds in \eqref{rh:jumpnorms}.

All formulas are holomorphic on each fixed open arc.  Their strict
geometric margins give a common open parameter neighborhood with
uniform bounds.  Scalar Cauchy estimates there give a Taylor series
convergent in $L^\infty$, hence also in $L^2$ since $\Sigma$ has finite
length.  This proves the asserted Banach-space holomorphy.
\end{proof}

Define the bounded operator
$\mathcal C_w f:=\mathcal C_-(fw)$ on
$L^2(\Sigma;\mathbb C^{2\times2})$.  The two preceding lemmas give
\[
 \|\mathcal C_w\|_{2\to2}\leq10^{16}/\Lambda_0\leq10^{-14}.
\]
The Neumann series therefore constructs uniquely
\begin{equation}\label{scalar:boundaryequation}
 \mu=I_2+(I-\mathcal C_w)^{-1}\mathcal C_-w,
 \qquad \mu-I_2=\mathcal C_-(\mu w),
\end{equation}
and gives
\begin{equation}\label{rh:muestimate}
 \|\mu-I_2\|_2\leq
 \frac{10^{17}\sqrt{b_0}}{\Lambda_0(1-10^{16}/\Lambda_0)}
 \leq\frac{2\cdot10^{17}\sqrt{b_0}}{\Lambda_0}.
\end{equation}
Since $t\mapsto\mathcal C_w$ is operator-norm holomorphic and its
norm is uniformly less than $1$, the uniformly convergent Neumann
series is holomorphic in operator norm. Thus $\mu-I_2$ is $L^2$-holomorphic.
The error matrix is
\begin{equation}\label{scalar:reconstruction}
 R(z;t):=I_2+\frac1{2\pi i}\int_\Sigma\frac{\mu(s;t)w(s;t)}{s-z}\,ds.
\end{equation}
It is holomorphic in $t$, locally uniformly for $z\notin\Sigma$,
since $\mu w=w+(\mu-I_2)w$ is $(L^1\cap L^2)$-holomorphic.
Plemelj and \eqref{scalar:boundaryequation} give
$R_-=\mu$ and $R_+=R_-(I_2+w)$, with $R(\infty)=I_2$.

For real $t$, positivity of the orthogonality measure gives the
original orthogonal-polynomial solution.  Its transformed error
matrix is bounded sectorwise at the artificial vertices, belongs
to the same $L^2$ boundary class, and has the Cauchy representation
\eqref{scalar:reconstruction}.  Its negative trace therefore satisfies
\eqref{scalar:boundaryequation}; the Neumann-series uniqueness
identifies it with $R$.  This avoids assuming complex solvability
before proving that the determinant has no zeros.

At $z=0$, \eqref{rh:cauchybound}, \eqref{rh:jumpnorms}, and
\eqref{rh:muestimate} yield
\begin{align}
 R(0)&=I_2+R_1+\mathcal E_R,\qquad
 R_1:=\frac1{2\pi i}\int_{\partial O_+\cup\partial O_-}
                       \frac{V_1(s)}s\,ds,\notag\\
 \|\mathcal E_R\|_{\mathrm F}
 &\leq\frac{\|w-V_1\|_1+\|\mu-I_2\|_2\|w\|_2}
               {2\pi\inf_{s\in\Sigma}|s|}
 <\frac{3\cdot10^{24}b_0}{\Lambda_0^2}.\label{rh:Rerror}
\end{align}

The first matching coefficient at $\xi_+$ is \cite[Lemma~3.1, (73)]{BBD}
\begin{equation*}
 V_1=\frac1{144k\alpha}
 \begin{pmatrix}
 -7\beta^2+5\beta^{-2}&-i(7\beta^2+5\beta^{-2})\\
 -i(7\beta^2+5\beta^{-2})&7\beta^2-5\beta^{-2}
 \end{pmatrix}.
\end{equation*}
It is meromorphic across the local curves, so its residue is
unchanged by the fixed-exit construction.  The scalar evaluations
\cite[(99)--(100)]{BBD}, with $\gamma=1+b^2$, are
\begin{align*}
 \frac1{2\pi i}\int_{\partial O_+}\frac{\beta^2}{\alpha z}\,dz
 &=-\frac3{4b^2}+\frac{3i}{4b},\\
 \frac1{2\pi i}\int_{\partial O_+}\frac{\beta^{-2}}{\alpha z}\,dz
 &=\frac3{4b^2}-\frac3{5\gamma}
   +i\left(\frac3{4b}-\frac{3b}{5\gamma}\right).
\end{align*}
These residues can also be evaluated directly on the selected local
branches. Put $w=x/b-1$ and $A:=4i\sqrt2\sqrt\gamma\,b^3/3$. With $c_b$
as above,
\[
 \alpha=Aw^{3/2}H(w,b^2),\qquad
 \sqrt{c_b}=(1-ib)/\sqrt\gamma,\qquad
 \frac{dz}{z}=\frac{-2ib\,dw}{1+b^2(1+w)^2}.
\]
Differentiation under the defining integral gives
\[
 H(0,b^2)=\gamma^{-2},\qquad
 \frac{H_w(0,b^2)}{H(0,b^2)}=\frac34-\frac{12b^2}{5\gamma}.
\]
The first integrand has simple-pole residue
$-3\sqrt{c_b}\sqrt\gamma/(4b^2)$.
The second is $F(w)\,dw/w^2$, with
\[
 F(0)=-\frac{3(1+ib)}{2b^2},\qquad
 \frac{F'(0)}{F(0)}=-\frac12+\frac{2b^2}{5\gamma}.
\]
These formulas give the two displayed evaluations, with no remaining
fractional powers or branch-cut contribution. They also verify the
first correction independently of any asymptotic error term.

For real $b$, the other endpoint contributes the conjugate matrix.
Their sum, equivalently \cite[Lemma~3.2, (92)]{BBD}, is
\begin{equation*}
 R_1=\left(\frac1{8kb^2}-\frac1{24k\gamma}\right)
              \begin{pmatrix}1&b\\b&-1\end{pmatrix}.
\end{equation*}
Both sides are holomorphic, so this equality continues to complex
$b$ by the identity theorem. For a matrix $A$, we write $(A)_{ij}$
for its entry in row $i$ and column $j$. In particular,
\begin{equation}\label{rh:cancellations}
 b(R_1)_{11}-(R_1)_{12}=0,\qquad
 (R_1)_{22}-b(R_1)_{21}=-\frac1{8\Delta}-\frac1{12k}.
\end{equation}

The exact reconstruction formulas \cite[(63)--(64)]{BBD} are
\begin{align}
 \frac{D_{k-1}}{D_k}
 &=\gamma^{-1/2}e^{k(-\gamma+\log\gamma+1)}
                    [R_{22}(0)-bR_{21}(0)],\label{rh:recovernorm}\\
 \pi_k(0;t)&=(-1)^k\gamma^{-1/2}
                    [bR_{11}(0)-R_{12}(0)].\label{rh:recoverreflection}
\end{align}
They hold first on the real interval and then meromorphically on
the disk, whereas their right sides are holomorphic. This meromorphic
continuation precedes the zero exclusion below and does not assume it.
Consequently,
\begin{equation*}
 \left|\pi_k(0;t)-(-1)^k\frac b{\sqrt\gamma}\right|
 \leq\frac{6\cdot10^{24}b_0}{\Lambda_0^2}.
\end{equation*}
Since $|b/\sqrt\gamma|\leq2b_0$ and
$b_0^2/\Lambda_0^2=\Delta_0^{-2}$, squaring gives
\eqref{rh:reflection}.

To justify the logarithms, use the radial Toda identity
\cite[Theorems~2.1(ii) and 2.5]{AVM} with $t_1=x$, $s_1=-y$,
and the Szeg\H{o} identity \cite[Corollary~2.8]{CT}.
The determinant indexed by $n$ in~\cite{CT} has size $n+1$, so it
corresponds to $D_{n+1}$ here. In meromorphic form, these identities give
\begin{equation}\label{rh:radialtoda}
 (\log D_k)''+t^{-1}(\log D_k)'
 =4D_{k+1}D_{k-1}/D_k^2=4(1-\pi_k(0;t)^2).
\end{equation}
The sign convention above makes the two-time identity positive, and
the factor $4$ follows by radial specialization:
$\partial_x\partial_y\log D_k=D_{k+1}D_{k-1}/D_k^2$,
$D_k(x,y)=D_k(\sqrt{xy})$, evaluated at $x=y=t$.
The right side of \eqref{rh:radialtoda} is holomorphic by
\eqref{rh:recoverreflection}.  A zero of $D_k$ of multiplicity $m$
would produce a double pole with coefficient $-m$ on the left;
$t^{-1}(\log D_k)'$ has at most a simple pole because $0$ lies
outside the disk.  Thus $D_k$ is zero-free.

By \eqref{rh:Rerror} and \eqref{rh:cancellations},
\[
 R_{22}(0)-bR_{21}(0)=1-\frac1{8\Delta}-\frac1{12k}+\varepsilon_R,
 \qquad |\varepsilon_R|\leq\frac{6\cdot10^{24}b_0}{\Lambda_0^2}.
\]
Also $|1/(8\Delta)+1/(12k)|\leq b_0/(6\Lambda_0)$, so this bracket
is nonzero and \eqref{rh:recovernorm} makes $D_{k-1}$ zero-free.
For the logarithm normalized at $t_0$, the inequality
$|\log(1+u)-u|\leq|u|^2/(1-|u|)$ gives
\[
 \left|\log\frac{D_k}{D_{k-1}}-
 \left[\Delta-\left(k-\frac12\right)\log\gamma
             +\frac1{8\Delta}+\frac1{12k}\right]\right|
 \leq\frac{9\cdot10^{24}b_0}{\Lambda_0^2}
 <\frac{10^{25}b_0}{\Lambda_0^2}.
\]
This proves \eqref{rh:norm}.  All geometric inequalities have strict
margins, so the construction and zero-free conclusions hold on a
neighborhood of the closed disk, completing
the lemma.
\end{proof}

On the positive real axis, write $\beta_j:=(-1)^j\pi_j(0;t)>0$.
At a real center $t$, put $\delta:=\Delta/(2t)$, with $\Delta$ as above.
Assume the hypotheses of Lemma~\ref{rh:quantitative} at center $t$ and index $k$,
with its real reflection bounds also valid for
$|j-k|\le\lfloor\Delta/4\rfloor$, and assume
\begin{equation}\label{rh:factorgates}
 \Delta\ge100,\quad0<\delta\le\frac14,\quad
 \Delta^3\ge64\cdot10^{26}t,\quad\Lambda^3\ge10^{27},
 \qquad\Lambda:=\Delta^{3/2}/\sqrt{k}.
\end{equation}
\begin{lemma}\label{lem:massfactor}
Let $B_k:=\log(1-\e^{-A_k})$, continued
from its real value. On the full disk,
\begin{equation}\label{rh:factorvalue}
 |A_k(z)-\Delta^2/(2k)|\le\Delta^2/(8k),\qquad
 \frac14\min(\Delta^2/k,1)\le|1-\e^{-A_k(z)}|
 \le2\min(\Delta^2/k,1),
\end{equation}
and $|B_k(z)-B_k(t)|<6$. On the half disk,
\begin{equation}\label{rh:factorderivatives}
 |B_k'|\le\frac{120000}{\Delta},\qquad
 |B_k''|\le\frac{4.8\cdot10^9}{\Delta^2}.
\end{equation}
At a common real radius,
\begin{equation}\label{der:eq:Bcurv}
 |2B_k-B_{k-1}-B_{k+1}|\le\frac{10^{29}}{\Delta^2}.
\end{equation}
\end{lemma}
\begin{proof}
Expanding the explicit term in \eqref{rh:norm} at $b^2=0$, and bounding
the logarithmic remainder by a geometric series, gives
\[
 \frac{|A_k-\Delta^2/(2k)|}{\Delta^2/k}
 \le0.000202+\frac{0.501}{\Delta}+\frac{0.126}{\Lambda^2}
      +\frac1{12\Delta^2}+\frac{10^{25}}{\Lambda^3}<\frac18.
\]
Thus $\Ree A_k\ge3\Delta^2/(8k)$. The integral formula for
$1-\e^{-A_k}$ gives its upper bound, and
$|1-\e^{-A_k}|\ge1-\e^{-\Ree A_k}$ gives the lower bound in
\eqref{rh:factorvalue}. Its real part is positive, so the continued
logarithm has oscillation at most $\log8+\pi<6$; Cauchy's inequalities
prove \eqref{rh:factorderivatives}.

For the real curvature, the reflection bounds and the discrete Painlev\'e
recurrence \eqref{low:dpII} give
\[
 \delta/2\le\beta_j^2\le2\delta,\qquad
 |\beta_{j+1}-2\beta_j+\beta_{j-1}|
 \le12\cdot10^{26}\sqrt\delta/\Delta^2.
\]
Summing over a segment of length $\lfloor\Delta/4\rfloor$ yields
$|\beta_k-\beta_{k-1}|\le28\cdot10^{26}\sqrt\delta/\Delta$.
The Szeg\H{o} identity $A_{k+1}-A_k=\log(1-\beta_k^2)$
\cite[Eqs.~(2.4) and (2.9)]{CT} consequently gives
\[
 |A_{k+1}-A_k|\le4\delta,\qquad
 |A_{k+1}-2A_k+A_{k-1}|\le84\cdot10^{26}/t.
\]
The neighboring values of $A$ are at least $\Delta^2/(8t)$.
For $v(x)=\log(1-\e^{-x})$, we have $|v'|\le x^{-1}$ and
$|v''|\le x^{-2}$. Taylor's formula therefore bounds the curvature of
$v(A_k)$ by $(672\cdot10^{26}+256)/\Delta^2\le10^{29}/\Delta^2$.
\end{proof}

Write $\mathscr L_j:=\log D_j$. For a real center $u$, put $d:=2u-j$ and define
\begin{equation}\label{der:detmodels}
\begin{aligned}
 V_j(z)&:=2jz-\frac{j^2}{2}\log\frac{2z}{j}-\frac{3j^2}{4},&
 P_j(z)&:=-\frac18\log\left(1-\frac{j}{2z}\right),\\
 W_j(z)&:=V_j(z)-\frac1{12}\log j+P_j(z)+\zeta'(-1),&
 E_j(z)&:=\log D_j(z)-W_j(z).
\end{aligned}
\end{equation}
Here $\zeta'(-1)$ is the derivative of the Riemann zeta function at $-1$.
Lemma~\ref{rh:quantitative} applies when $b^2=d/j\le10^{-12}$ and
$\Lambda=d^{3/2}/\sqrt j\ge10^{30}$. Under~\eqref{rh:factorgates} and its reflection-window hypothesis,
the derivative and curvature bounds for $B_j$ are those of
Lemma~\ref{lem:massfactor}, with $(k,t,\Delta)$ replaced by $(j,u,d)$.
Section~\ref{left:sec:transfer} verifies these hypotheses at each saddle.

\begin{lemma} \label{der:lem:derivative}
Let $C\ge1$, $\Delta=2t-k\ge32$, $0<\delta=\Delta/(2t)\le1/4$,
and $\Delta^3\ge64Ct$. Suppose
\[
 |\beta_q^2-\delta_q|\le C\Delta_q^{-2},\qquad
 \Delta_q=2t-q,\quad\delta_q=\Delta_q/(2t),\qquad
 |q-k|\le\lfloor\Delta/4\rfloor.
\]
Then $|E_k'(t)|\le37C/\Delta$. If the corresponding complex reflection
bound and zero-freeness hold on $|z-t|\le a\Delta$, $0<a\le10^{-4}$,
then
\[
 |E_k'(z)|<75C/\Delta
 \quad\text{on the full disk},\qquad
 |E_k''(z)|\le150C/(a\Delta^2)
 \quad\text{on the half disk}.
\]
\end{lemma}
\begin{proof}
The hypotheses give $\delta/2\le\beta_q^2\le2\delta$ throughout the
window. The exact recurrence
$t(1-\beta_q^2)(\beta_{q+1}+\beta_{q-1})=q\beta_q$ implies
\begin{equation*}
 |\beta_{q+1}-2\beta_q+\beta_{q-1}|
 \le12C\sqrt\delta/\Delta^2.
\end{equation*}
To avoid a quadratic dependence on $C$, interpolate over
$m=\lfloor\Delta/(4\sqrt C)\rfloor$. Since $\Delta^2\ge128C$,
\[
 2\le m,\qquad\frac\Delta{6\sqrt C}\le m\le\frac\Delta{4\sqrt C}.
\]
The identity
$f_m-f_0=m(f_1-f_0)+\sum_{r=1}^{m-1}(m-r)(f_{r+1}-2f_r+f_{r-1})$,
applied to $f_r=\beta_{k-1+r}$, together with its oscillation
at most $\sqrt{\delta/2}$, yields
\begin{equation}\label{der:eq:beta1}
 |\beta_k-\beta_{k-1}|<6\sqrt{C\delta}/\Delta.
\end{equation}
All recurrence indices remain inside the assumed window.

The key point is the exact finite-boundary identity
\begin{equation}\label{der:eq:boundary}
 \mathscr L_k'-V_k'=2t\left[(\beta_k\beta_{k-1}-\delta)^2
                         -(\beta_k-\beta_{k-1})^2\right].
\end{equation}
Indeed, differentiating the orthogonality norms makes the polynomial
derivatives vanish and gives $\mathscr L_k'=2\sum_{j=1}^k\beta_j\beta_{j-1}$.
The discrete recurrence telescopes this sum through
\[
 H_j:=t^2[(1-\beta_j^2)(1-\beta_{j-1}^2)-1]+tj\beta_j\beta_{j-1},
 \qquad H_{j+1}-H_j=t\beta_j\beta_{j+1}.
\]
Using $\beta_0=1$ yields
$\mathscr L_k'=2t(1-\beta_k^2)(1-\beta_{k-1}^2)+2k\beta_k\beta_{k-1}$,
which is \eqref{der:eq:boundary}.

Positivity bounds the product between the two squares, so
$|\beta_k\beta_{k-1}-\delta|\le(2t)^{-1}+2C/\Delta^2$.
Equations \eqref{der:eq:beta1}--\eqref{der:eq:boundary} now give
\[
 |\mathscr L_k'-V_k'|\le\frac1t+\frac{16C^2t}{\Delta^4}+\frac{36C}{\Delta}.
\]
Since $t\ge2\Delta$, $\Delta^3\ge64Ct$, and
$|P_k'|\le1/(4\Delta)$, this proves $|E_k'|\le37C/\Delta$.

For complex propagation, radial Toda gives
\[
 E_k''(z)+z^{-1}E_k'(z)=\mathcal F(z),\qquad
 |\mathcal F(z)|\le(4C+3)/\Delta^2.
\]
Integrating $(zE_k')'=z\mathcal F$ from $t$ to $z$, with
$|t/z|,|w/z|<2$, gives
$|E_k'(z)|\le(74+14a)C/\Delta<75C/\Delta$.
Cauchy's inequality on the half disk proves the second-derivative bound.
\end{proof}

The exact saddle remainder is
\[
 \mathscr R_j:=-\frac1{12}\log j+P_j+\zeta'(-1)+E_j+B_j,
 \qquad\log(D_j-D_{j-1})=V_j+\mathscr R_j.
\]
Combining $|P_j'|\le d^{-1}$, $|P_j''|\le10d^{-2}$ with the preceding
bounds gives
\begin{equation}\label{der:eq:Kconstants}
 |\mathscr R_j'|\le10^{28}/d,\qquad |\mathscr R_j''|\le2\cdot10^{32}/d^2.
\end{equation}
No bound on the additive value $E_j(u)$ is used.

\section{Proof of the main theorem}\label{sec:proof}

For $n\ge N$ and $2\le k\le n-1$, we use the following five overlapping
regions, listed below.
\begingroup
\setlength{\leftmargini}{4em}
\begin{enumerate}
\setlength{\itemsep}{0pt}\setlength{\parsep}{0pt}
\item[Case 1.] The lower region: $k\le(2-10^{-13})\sqrt n$.
\item[Case 2.] The upper region: $k\ge2\sqrt n+10n^{1/6}(\log n)^{2/3}$.
\item[Case 3.] The compact edge window: $|s|\le2\cdot10^{20}+1$.
\item[Case 4.] The right logarithmic window: $23\le s\le11(\log n)^{2/3}$.
\item[Case 5.] The drifting left region: $(2-10^{-13})\sqrt n\le k<2\sqrt n$ and $T\ge2\cdot10^{20}$.
\end{enumerate}
\endgroup
Each estimate includes the neighbors needed for log-concavity. The final
assembly verifies coverage and the regional size thresholds, and handles
$k=1,n$.

\subsection{Case 1: the lower region}\label{low:section}
The proof uses a real-axis comparison for discrete Painlev\'e II,
analytic interpolation for coefficient extraction, and an exact
Barnes-function cancellation for the curvature.

For $t>0$ and $k\ge1$, the reflection coefficients satisfy
\begin{align}
 0&<\beta_k(t)<1,\label{low:positive}\\
 (1-\beta_k^2)(\beta_{k-1}+\beta_{k+1})&=(k/t)\beta_k.
 \label{low:dpII}
\end{align}
The shifted-determinant numerator of $\beta_k$ is
$\det[I_{1+i-j}(2t)]_{i,j=0}^{k-1}=\sum_{\ell(\lambda)\le k}c_\lambda c_{\lambda+(1^k)}>0$
by Schur orthogonality, where $c_\lambda:=f^\lambda t^{|\lambda|}/|\lambda|!$.
This gives the positivity in~\eqref{low:positive}.
The upper bound follows from the Szeg\H{o} norm identity
\cite[Eq.~(2.9)]{CT}, and the recurrence is
\cite[Eq.~(1.8)]{CT}, with $\theta_1=t$ and $x_k=(-1)^k\beta_k$.
We also use the radial Toda identity~\eqref{rh:radialtoda}.
The entire function $F_k$ in~\eqref{pre:det-def} has the nonnegative
coefficient expansion~\eqref{pre:gessel}. Taking absolute values
in~\eqref{pre:angular-integral} at $t=\sqrt r\,e^{i\theta/2}$ gives
\begin{equation}\label{low:haar-modulus}
 |F_k(re^{i\theta})|\le F_k(r\cos^2(\theta/2))
 \qquad(r>0,\ |\theta|\le\pi).
\end{equation}

\begin{lemma} \label{low:beta-expansion}
If $t\ge10^{200}$ and $1\le k\le2t(1-10^{-16})$, then, with $x=k/(2t)$,
\begin{equation}\label{low:beta-bound}
 \left|\beta_k(t)^2-\left(1-x-\frac{x}{32t^2(1-x)^2}\right)\right|
 \le10^{118}t^{-4}.
\end{equation}
\end{lemma}
\begin{proof}
Put $K:=\lfloor2t(1-5\cdot10^{-17})\rfloor$ and
\[
 A(x):=\sqrt{1-x}-\frac{x}{64t^2(1-x)^{5/2}},\qquad A_j:=A(j/(2t)).
\]
On the comparison grid,
\begin{equation}\label{low:comparison-range}
 5\cdot10^{-9}\le A_j\le1\qquad(0\le j\le K).
\end{equation}
Taylor expansion on $0\le x\le1-2.5\cdot10^{-17}$ gives the residual estimate
\begin{equation}\label{low:residual}
 |\mathcal R_j|\le10^{92}t^{-4},\qquad
 \mathcal R_j:=(1-A_j^2)(A_{j-1}+A_{j+1})-(j/t)A_j.
\end{equation}
For transparency, if $f=\sqrt{1-x}$ and
$g=-x/[16(1-x)^{5/2}]$, then $A=f+g/(4t^2)$ and the order-$t^{-2}$
term cancels by $-4f^2g+xf''=0$.
The uniform bounds
\[
 |g|\le10^{45},\quad |f''|\le10^{27},\quad |f^{(4)}|\le10^{63},
 \quad |g''|\le10^{81},\quad |g^{(4)}|\le10^{119}
\]
bound the remaining terms by the right side of \eqref{low:residual}.

For $e_j:=\beta_j-A_j$, exact subtraction gives
\begin{equation*}
 \lambda_je_j-(1-\beta_j^2)(e_{j-1}+e_{j+1})=\mathcal R_j,
 \qquad\lambda_j:=j/t+(\beta_j+A_j)(A_{j-1}+A_{j+1}).
\end{equation*}
To bound $\beta_j$ from below, consider $t>0$ and $j\ge0$.
The finitely supported sequence $v_j:=\max\{1-j/(2t),0\}$ satisfies
$v_0=\beta_0=1$ and
\[
 (1-v_j^2)(v_{j-1}+v_{j+1})-(j/t)v_j\ge2v_j^2(1-v_j)\ge0.
\]
An interior maximum $v_j/\beta_j>1$ would give
$(v_{j-1}+v_{j+1})/v_j\le(\beta_{j-1}+\beta_{j+1})/\beta_j$,
contradicting this inequality and \eqref{low:dpII}.
Thus $\beta_j(t)\ge\max\{1-j/(2t),0\}$, and in particular
$\beta_j\ge5\cdot10^{-17}$ on the comparison grid. The same derivative
bounds yield the strict diagonal gap
\begin{equation*}
 \lambda_j-2(1-\beta_j^2)
 \ge5\cdot10^{-25}-t^{-2}(10^{45}+10^{27})>10^{-25}.
\end{equation*}
Consequently $10^{117}t^{-4}+e^{-10^{-13}(K-j)}$ is a supersolution
for $|e_j|$: its constant part absorbs \eqref{low:residual}, and its
exponential part has nonnegative image because
$2\cosh(10^{-13})-2<10^{-25}$.
It also dominates $|e_0|=0$ and $|e_K|\le1$. The finite-interval
maximum principle therefore bounds $|e_j|$ by this supersolution.
For the stated $k$, $K-k\ge10^{-16}t-1$ and
$e^{10^{-13}-10^{-29}t}\le t^{-4}$ for $t\ge10^{200}$.
Thus $|\beta_k-A_k|\le2\cdot10^{117}t^{-4}$.
Squaring $A_k$ and using \eqref{low:comparison-range} proves
\eqref{low:beta-bound}.
\end{proof}

For integer $k\ge1$, write
\[
 \mathcal G(k+1):=\prod_{j=0}^{k-1}j!,\qquad
 \mathcal V_k(t):=2kt-\frac{k^2}{2}\log(2t)+\log\mathcal G(k+1)
                          -\frac k2\log(2\pi).
\]
We use exact values of the Barnes $G$-function, not its asymptotic expansion.
\begin{lemma} \label{low:determinant}
If $t\ge10^{200}$ and $1\le k\le2t(1-10^{-16})$, then
\begin{equation}\label{low:determinant-bound}
 \left|\log D_k(t)-\mathcal V_k(t)+\frac18\log\left(1-\frac{k}{2t}\right)\right|
 \le10^{118}t^{-2}.
\end{equation}
\end{lemma}
\begin{proof}
Set $\mathcal E_k:=\log D_k-\mathcal V_k+\tfrac18\log(1-k/(2t))$.
The Toda and reflection estimates give
\begin{equation*}
 |\mathcal E_k''+t^{-1}\mathcal E_k'|\le4\cdot10^{118}t^{-4}.
\end{equation*}
The fixed-index normalization \cite[Eqs.~(41)--(43)]{BBD} is
\[
 D_k(t)\sim e^{2kt}(2t)^{-k^2/2}(2\pi)^{-k/2}\mathcal G(k+1)
 \qquad(t\to\infty,\ k\text{ fixed}),
\]
so $\mathcal E_k(t)\to0$; this fixes the integration constant.
Writing $R_k:=\mathcal E_k''+t^{-1}\mathcal E_k'$, the finite limit at infinity and the
integrability of $tR_k(t)$ imply the exact representation
\[
 \mathcal E_k(t)=\int_t^\infty sR_k(s)\log(s/t)\,ds.
\]
Since $\int_t^\infty s^{-3}\log(s/t)\,ds=(4t^2)^{-1}$, this proves
\eqref{low:determinant-bound} with no unspecified normalization.
\end{proof}

\begin{lemma}\label{low:saddle}
Put $c_*:=2-5\cdot10^{-14}$. For $n\ge10^{3000}$ and
$1\le k\le c_*\sqrt n$, put
$t:=n/k+k/4$ and $b:=(n+k^2/4)/2$. Then
\begin{equation}\label{low:coefficient}
 a_{n,k}=\frac{n!D_k(t)}{t^{2n}\sqrt{2\pi b}}(1+\epsilon_{n,k}),
 \qquad |\epsilon_{n,k}|\le10^{-21}.
\end{equation}
\end{lemma}
\begin{proof}
With $c=k/\sqrt n$, we have $t\ge\sqrt n$,
$n/2\le b\le n$, $kt=2b$, and
\begin{equation*}
 \delta=1-\frac{k}{2t}=\frac{4-c^2}{4+c^2}>2.4\cdot10^{-14}>2\cdot10^{-15}.
\end{equation*}
Define
$\Phi(w):=F_k(t^2e^{w/\sqrt b})F_k(t^2)^{-1}e^{-nw/\sqrt b}$.
For real $|u|\le40$, the radius $te^{u/(2\sqrt b)}$ remains in the
range of Lemma~\ref{low:determinant}. The saddle cancellation gives
\[
 2kt(e^{u/(2\sqrt b)}-1)-\frac{k^2u}{4\sqrt b}-\frac{nu}{\sqrt b}
 =\frac{u^2}{2}+R_{\mathrm{ph}},\qquad
 |R_{\mathrm{ph}}|\le10^5/\sqrt n.
\]
The logarithmic prefactor changes by at most $10^{18}/\sqrt n$ and the
two determinant errors sum to at most $4\cdot10^{118}/n$. Hence
\begin{equation*}
 |\log\Phi(u)-u^2/2|\le10^{20}/\sqrt n\qquad(|u|\le40).
\end{equation*}
Nonnegative coefficients give $|\Phi(u+iv)|\le\Phi(u)$.
Apply \eqref{pre:twoconstants} to $\Phi(w)-e^{w^2/2}$ with
$M=3e^{800}$ and $\varepsilon=2\cdot10^{20}e^{800}/\sqrt n$:
\begin{equation*}
 |\Phi(w)-e^{w^2/2}|\le10^8e^{800}n^{-1/6}<10^{-100}
 \qquad(|w|\le20).
\end{equation*}

For $|\theta|\le\theta_0=10^{-8}$, the reduced radius
$t\cos(\theta/2)$ still has gap at least $10^{-16}$.
The derivative of $\mathcal V_k(s)$ with respect to $\log s$ is at least $n$
between these radii; the remaining prefactor ratio is at most $2$.
Thus \eqref{low:haar-modulus} gives
\begin{equation*}
 |\Phi(i\sqrt b\,\theta)|\le2e^{-n\theta^2/8}
 \le2e^{-b\theta^2/8}.
\end{equation*}
For $\theta_0\le|\theta|\le\pi$, monotonicity extends the bound
$2e^{-10^{-16}n/8}$. In the coordinate $y=\sqrt b\,\theta$, the integral
outside $|y|\le20$ is therefore at most
\[
 4\int_{20}^\infty e^{-y^2/8}\,dy+4\pi\sqrt n\,e^{-10^{-16}n/8}
 \le\frac45e^{-50}+4\pi\sqrt n\,e^{-10^{-16}n/8}<2\cdot10^{-22}.
\]
The Gaussian tail is below $10^{-80}$. Substitution in the exact formula
\[
 a_{n,k}=\frac{n!F_k(t^2)}{2\pi t^{2n}\sqrt b}
 \int_{-\pi\sqrt b}^{\pi\sqrt b}\Phi(iy)\,dy
\]
proves \eqref{low:coefficient}.
\end{proof}

Define
\begin{align}
 H(c)&:=1-\frac{c^2}{4}-\frac{c^2}{2}\log\frac{4+c^2}{2c^2}
                       -2\log\frac{4+c^2}{4c},\label{low:H}\\
 B(x)&:=\frac{x^2}{2}\log x-\frac{3x^2}{4}+\frac x2\log(2\pi),
 \notag\\
 Z(k)&:=\log\mathcal G(k+1)-B(k),\notag\\
 J(c)&:=-\frac18\log\frac{4-c^2}{4+c^2}
                   -\frac12\log\frac{1+c^2/4}{2}.\notag
\end{align}
The determinant and coefficient estimates give, for $n\ge10^{3000}$ and
$1\le k\le c_*\sqrt n$,
\begin{equation}\label{low:log-cdf}
 \log a_{n,k}=C_n+nH(k/\sqrt n)+Z(k)+J(k/\sqrt n)+\rho_{n,k},
 \qquad|\rho_{n,k}|\le3\cdot10^{-21},
\end{equation}
where $C_n:=\log(n!)+n-n\log n-\tfrac12\log(2\pi n)$ is independent of $k$.

\begin{lemma} \label{low:curvature}
For $n\ge10^{3000}$, $k\ge10^{10}$ and $k+1\le c_*\sqrt n$,
\begin{equation}\label{low:cdf-curvature}
 \log\frac{a_{n,k}^2}{a_{n,k-1}a_{n,k+1}}>10^{-15}.
\end{equation}
For $2\le j\le c_*\sqrt n$,
\begin{equation}\label{low:cdf-ratio}
 \frac{a_{n,j-1}}{a_{n,j}}
 \le\exp(-\tfrac12\,10^{-30}\sqrt n)<10^{-21}.
\end{equation}
\end{lemma}
\begin{proof}
Apply the second-difference formula~\eqref{pre:peano} with $h=1$.
For a real variable $0<c\le c_*$, write $\delta=(4-c^2)/(4+c^2)$.
Differentiation gives
\begin{equation*}
 -H''(c)=-\log(1-\delta)+\frac{\delta^2}{1-\delta}
 \ge\delta\ge2\cdot10^{-15}\qquad(0<c\le c_*).
\end{equation*}
The exact Barnes recurrence \cite[Eq.~(44)]{BBD} and $B''(x)=\log x$ give
\begin{equation*}
 2Z(k)-Z(k-1)-Z(k+1)
 =\int_{-1}^1(1-|v|)\log\frac{k+v}{k}\,dv
 \ge-\frac1{12(k-1)^2}.
\end{equation*}
Also $|J''(c)|\le10^{32}$ on $[0,c_*]$. Consequently
\[
 \log\frac{a_{n,k}^2}{a_{n,k-1}a_{n,k+1}}
 \ge2\cdot10^{-15}-\frac1{12(k-1)^2}-\frac{10^{32}}n
                                      -12\cdot10^{-21}>10^{-15}.
\]

For the ratios, $H'(c)/c=y-\log(1+y)$, where
$y=\delta/(1-\delta)$. Since
$y-\log(1+y)\ge y^2/[2(1+y)]\ge\delta^2/2$,
and $H'$ is decreasing, $H'(c)\ge10^{-30}$ on $(0,c_*]$.
The Barnes recurrence and integral bounds for $\log\Gamma(j)$ give
$|Z(j)-Z(j-1)|\le2\log(j+1)$; directly, $|J'|\le10^{16}$.
Taking a first difference in \eqref{low:log-cdf} yields
\[
 \log\frac{a_{n,j-1}}{a_{n,j}}
 \le-10^{-30}\sqrt n+2\log(j+1)+10^{16}/\sqrt n+6\cdot10^{-21}
 \le-\tfrac12\,10^{-30}\sqrt n.
\]
The last comparison already holds at $n=10^{80}$ and improves thereafter.
\end{proof}

\begin{lemma} \label{low:small-k}
For $k\ge2$ and $n\ge k^6$, $p_{n,k}^2>p_{n,k-1}p_{n,k+1}$.
\end{lemma}
\begin{proof}
Let $U_k:=n!a_{n,k}$ and $M_k:=\binom{n+k^2-1}{n}$.
Transposing tableaux in~\eqref{eq:rsk-count} gives
$U_k=\sum_{\lambda\vdash n,\,\ell(\lambda)\le k}(f^\lambda)^2$.
Write $s_\lambda(1^k)$ for the number of semistandard tableaux of shape
$\lambda$ with entries in $[k]$, weakly increasing along rows and
strictly increasing down columns. The word and matrix forms of the
RSK correspondence~\cite{Knuth} give
\[
 k^n=\sum_{\substack{\lambda\vdash n\\\ell(\lambda)\le k}}
 f^\lambda s_\lambda(1^k),\qquad
 M_k=\sum_{\substack{\lambda\vdash n\\\ell(\lambda)\le k}}
 s_\lambda(1^k)^2.
\]
The first identity counts words of length $n$ on $[k]$; the second counts
$k\times k$ nonnegative integer matrices with entry sum $n$.
Cauchy--Schwarz gives $k^{2n}\le U_kM_k$, while encoding each standard
tableau by its row word gives $U_k\le k^{2n}$.
Thus $P_{n,k}\ge k^{2n}/M_k-(k-1)^{2n}$ and
$P_{n,k-1}P_{n,k+1}\le(k^2-1)^{2n}$, so it suffices that
\begin{equation*}
 \left(\frac{k^2}{k^2-1}\right)^n\ge2M_k.
\end{equation*}
Using $M_k\le(n+k^2)^{k^2-1}$, this follows from
$$n/k^2>\log2+(k^2-1)\log(n+k^2).$$
The difference is increasing for $n\ge k^6$ and is positive at $n=k^6$:
for $k\ge3$ use $\log(2k^6)<k^2$; for $k=2$ use
$16-3\log68-\log2>0$.
\end{proof}

\begin{lemma} \label{low:theorem}
For $n\ge10^{3000}$ and $1\le k\le(2-10^{-13})\sqrt n$,
$p_{n,k}^2>p_{n,k-1}p_{n,k+1}$.
\end{lemma}
\begin{proof}
If $k\ge10^{10}$ lies in the claimed range, then $k+1\le c_*\sqrt n$.
Equation \eqref{low:cdf-ratio} implies
$|\log p_{n,j}-\log a_{n,j}|<2\cdot10^{-21}$ for $j=k-1,k,k+1$.
Hence \eqref{low:cdf-curvature} gives
$\mathcal C_{n,k}>10^{-15}-8\cdot10^{-21}>0$.
For $2\le k<10^{10}$ use Lemma~\ref{low:small-k}; at $k=1$ use
$p_{n,0}=0<p_{n,1}=1/n!$.
\end{proof}

\subsection{Case 2: the upper region}\label{up:section}
Throughout this region, write $r:=n-k$.

Define the Charlier polynomials and coefficients $d_i$ by
\begin{equation}\label{up:charlier-def}
 \sum_{j\ge0}C_j(x)\frac{v^j}{j!}=e^{-v}(1+v)^x,
 \qquad \sum_{i\ge0}d_iv^i=\frac{e^{-v}}{1-v}.
\end{equation}
For nonnegative integers $n,r$, set
\begin{equation*}
 U_r(n):=[v^r]\sum_{j=0}^n
 \frac{(n!)^2}{j!(n-j)!^2}v^j(1-v)^{2(n-j)-1},
 \qquad A_j:=\frac{C_j(n)^2}{j!}.
\end{equation*}
Set $U_r(n):=0$ for $r<0$. For fixed $n,k$, write
\[
 S:=\sum_{\substack{\lambda\vdash n\\\lambda_2\ge k+1}}(f^\lambda)^2.
\]
We have $d_0=1$, $d_1=0$, $0\le d_i\le1$ and $d_i\ge1/3$ for
$i\ge2$. The standard Charlier recurrence
\cite[(18.22.2), Table~18.22.1]{DLMF}, in this normalization, is
\begin{equation}\label{up:charlier-rec}
 C_0(n)=1,\quad C_1(n)=n-1,\quad
 C_j(n)=(n-j)C_{j-1}(n)-(j-1)C_{j-2}(n).
\end{equation}

\begin{lemma}
\label{up:occupancy}
For $n\geq 1$, $0\le k\le n$ and   $r=n-k$,
\begin{align}
 U_r(n)&=\sum_{\substack{\lambda\vdash n\\
                  \lambda_i=k+i-1\text{ for some }i}}(f^\lambda)^2,
                  \label{up:particle}\\
 U_j(n)&=\sum_{i=0}^j d_iA_{j-i}\qquad(j\ge0),
                  \label{up:convolution}\\
 P_{n,k}&\ge U_r-S,\qquad
 P_{n,k-1}\le U_{r+1},\qquad P_{n,k+1}\le U_{r-1}.
                  \label{up:transfer-basic}
\end{align}
If $n<2k+2$, then $U_r=P_{n,k}$.
\end{lemma}
\begin{proof}
Let $s_\alpha$ be the Jacobi--Trudi determinant of an integer sequence,
and let $\operatorname{ex}$ be the specialization $p_1=1$, $p_j=0$
for $j>1$. The Bernstein identities \cite[Chapter~I, \S5, Example~29(d), pp.~95--96]{Macdonald}
\[
 \begin{gathered}
 B_X(z):=\sum_{a\in\mathbb Z}z^a\sum_{j\ge0}(-1)^jh_{a+j}(X)e_j^\perp,\\
 [z^a]B_X(z)s_\nu=s_{(a,\nu)},\qquad
 B_X(z)g(X)=H_X(z)g(X-z^{-1})
 \end{gathered}
\]
and the Cauchy identity give
\begin{equation*}
 \sum_{m,a,b}t^mz^aw^b\sum_{\nu\vdash m}
 \operatorname{ex}(s_{(a,\nu)})\operatorname{ex}(s_{(b,\nu)})
 =\frac{\exp(z+w+t-t/z-t/w)}{1-t/(zw)}.
\end{equation*}
Substituting $t=vzw$ and extracting $z^nw^n$ yields
\[
 \sum_{m\ge0}(n!)^2\sum_{\nu\vdash m}
 \operatorname{ex}(s_{(n-m,\nu)})^2v^m
 =\sum_{j=0}^n\frac{(n!)^2v^j(1-v)^{2(n-j)-1}}{j!(n-j)!^2}.
\]
Straightening $s_{(k,\nu)}$ sorts the shifted coordinates
$k-1,\nu_1-2,\nu_2-3,\ldots$. Repeated coordinates give zero;
otherwise the position $i$ of $k-1$ gives $\pm s_\lambda$ with
$\lambda_i=k+i-1$. This is a bijection, with inverse
$\nu=(\lambda_1+1,\ldots,\lambda_{i-1}+1,\lambda_{i+1},\ldots)$.
Since $\operatorname{ex}(s_\lambda)=f^\lambda/n!$, this proves
\eqref{up:particle}. RSK identifies its $i=1$ terms with $P_{n,k}$;
every other term has $\lambda_2\ge k+1$ and $n\ge2k+2$.

For the convolution, the classical bilinear Charlier identity
\cite[\S4, (3)]{GesselJayawant}, in the normalization
\eqref{up:charlier-def}, is
\[
 \sum_{j\ge0}\frac{C_j(n)C_j(m)}{j!}v^j
 =e^v(1-v)^{n+m}\sum_{a=0}^{\min(n,m)}
 \frac{n^{\underline a}m^{\underline a}}{a!}
 \left(\frac{v}{(1-v)^2}\right)^a.
\]
Taking $m=n$ and multiplying by $e^{-v}/(1-v)$ gives
\eqref{up:convolution}. The stated bounds on $d_i$ follow from $d_i=\sum_{j=0}^i(-1)^j/j!$.
\end{proof}

\begin{lemma}\label{up:bridge-region}
If $h:=k-1\ge64$ and $n\le h^2/15$, then
$P_{n,k}^2>P_{n,k-1}P_{n,k+1}$ at every interior index $1<k<n$.
In particular this conclusion holds whenever $k\ge1000$ and
$k^2\ge16(n-k)$.
\end{lemma}
\begin{proof}
We use $U_j,A_j,d_j$ from Lemma~\ref{up:occupancy}; in particular
$0\le d_j\le1/2$ for $j\ge2$, by their alternating-series formula.
Put $R_j:=C_j(n)/C_{j-1}(n)$.
Induction in \eqref{up:charlier-rec}, through index $r+1$, gives
\[
 (n-j)/2\le R_j\le n-j,\qquad
 1<D_j:=R_{j-1}-R_j\le2\quad(j\ge2).
\]
For the first induction, with $v=n-j\ge h$, use
$R_j\ge v-2(j-1)/(v+1)\ge v/2$, since $4n\le h^2$.
For the second use
\[
 D_2=1+1/R_1,\qquad
 D_j=1+1/R_{j-1}+
       (j-2)D_{j-1}/(R_{j-1}R_{j-2}).
\]
Its upper bound is $1+2/(h+1)+8/15<2$.
All ratios are therefore positive before they are used.

Set $t_j:=A_{j-1}/A_j=j/R_j^2$ for $j\ge1$, and $t_j=0$ for
$j\le0$. They increase, and the derivative bound for $x^{-2}$ on
$x\ge h/2$ gives
\[
 t_j\le q:=4/15,\qquad
 0\le t_j-t_{j-1}\le\delta:=4/h^2+32n/h^3
                  \le4/h^2+32/(15h).
\]
For $H_j:=U_j/A_j$, the convolution becomes
$H_j=1+\sum_{i\ge2}d_i\prod_{\ell=0}^{i-1}t_{j-\ell}$.
Each product increases; telescoping successive products gives
\[
 0\le H_{j+1}-H_j\le\frac\delta2\sum_{i\ge2}iq^{i-1}
 =\frac{\delta q(2-q)}{2(1-q)^2}
 \le\frac{1664}{1815h}+\frac{208}{121h^2}<\frac1{h+1}.
\]
The last inequality is equivalent to
$151h^2-4784h-3120>0$, true at $h=64$ and increasing thereafter.
Since $H_j\ge1$ and $H_{r-1}\le H_r$, while
$A_r^2/(A_{r-1}A_{r+1})>1+2/h$, we obtain
\begin{equation}\label{up:bridge-curvature}
 \frac{U_r^2}{U_{r-1}U_{r+1}}
 >\frac{1+2/h}{1+1/(h+1)}=\frac{k}{k-1}.
\end{equation}

Put $Q:=U_r-P_{n,k}\ge0$. If $Q=0$, the neighboring bounds in
\eqref{up:transfer-basic} prove the assertion. Otherwise
$n\ge2k+2$. For $\lambda=(a,b,\tau)$ with $a\ge b\ge k+1$ and
$m=|\tau|$, the hook-length formula gives
$f^{(a,b,\tau)}\le n!f^\tau/(a!b!m!)$.
Sum $\sum_{\tau\vdash m}(f^\tau)^2=m!$ and drop the shape restrictions.
Successive increments of $a$ or $b$ multiply the resulting summand
by at most $(n-2k-2)/(k+2)^2<1/15$. The two geometric series yield
\[
 Q\le\sum_{\lambda_2\ge k+1}(f^\lambda)^2
 \le\frac{4(n!)^2}{((k+1)!)^4(n-2k-2)!}.
\]
The ratio lower bound further gives, for $j\le r$,
\[
 R_j\ge(n-j)(1-u_j),\qquad
 u_j:=\frac{2j}{(n-j)^2}<\frac2{15},\qquad
 \sum_{j=1}^r u_j\le2n\sum_{v=k}^{\infty}v^{-2}\le\frac{2n}h.
\]
Use $\log(1-u)\ge-15u/13$ and multiply, obtaining
\[
 A_r\ge\frac1{r!}\left(\frac{(n-1)!}{(k-1)!}\right)^2
                   e^{-60n/(13h)}.
\]
The bounds $r!/(r-k-2)!\le n^{k+2}$ and
$m!\ge(m/e)^m$ now give
\[
 \frac Q{A_r}\le\frac{4n^3}{k^2(k+1)^2}
 \left(\frac{e^2n}{(k+1)^2}\right)^{k+1}e^{60n/(13h)}.
\]
Since $n\le h^2/15<k^2/15$, $e^2<8$, and
$e^{60n/(13h)}\le e^{4k/13}\le(13/9)^k$, this is less than
\[
 \frac{32k^2}{50625}\left(\frac{104}{135}\right)^k
 <\frac{k^2}{1500}\left(\frac45\right)^k<\frac1{100k}.
\]
For the last step, $k^3(4/5)^k$ decreases for integers $k\ge64$,
since $(4/5)(65/64)^3<1$. At $64$ use $(4/5)^8<1/4$ to get
$64^3(4/5)^{64}<4<15$. Thus $Q/U_r\le Q/A_r<1/(100k)$.
Combining this with \eqref{up:bridge-curvature} and
\eqref{up:transfer-basic} proves strict log-concavity because
$(1-1/(100k))^2k/(k-1)>1$.

Finally, if $k\ge1000$ and $k^2\ge16(n-k)$, then $h\ge999$ and
\[
 n\le k^2/16+k=h^2/16+9h/8+17/16\le h^2/15.
\]
The last inequality is equivalent to $h^2-270h-255\ge0$, valid
for $h\ge999$. This establishes the stated bridge.
\end{proof}

\begin{lemma}
\label{up:near-curvature}
Suppose
\begin{equation}\label{up:near-domain}
 k>2\sqrt n,\qquad \frac{k^2}{16}<r<\frac{(k-1)^2}{4},\qquad
 \eta:=\frac{\sqrt{(k-1)^2-4r}}{k-1},\quad
 \varepsilon:=\min(\eta,1/12),\quad k\varepsilon^3\ge10^7.
\end{equation}
Then
\begin{equation*}
 \log\frac{U_r^2}{U_{r-1}U_{r+1}}\ge\frac1{16k\varepsilon}.
\end{equation*}
\end{lemma}
\begin{proof}
Put $R_j:=C_j(n)/C_{j-1}(n)$, $\Delta_j:=R_j-R_{j+1}$,
and $a_j:=A_j/A_{j-1}=R_j^2/j$. Until
the final comparison it suffices that $k\varepsilon^3\ge100$; in
particular $k\ge172800$ and $k<4\sqrt n$. Define the larger roots
\[
 \rho_j:=\frac{n-j+\sqrt{(n-j)^2-4(j-1)}}2
 \qquad(1\le j\le r+1).
\]
They decrease with $j$, and the Charlier recurrence gives
\begin{equation}\label{up:root-bounds}
 \rho_j\le R_j\le n-j,\qquad R_j\ge(k-1)/2\ge k/3.
\end{equation}
The relative discriminants are at least $\eta$, so successive roots
differ by at most $2/\varepsilon$. In
\[
 R_j-\rho_j=\frac{j-1}{\rho_jR_{j-1}}
                   (R_{j-1}-\rho_j),
\]
the multiplier is at most $(1-\varepsilon)/(1+\varepsilon)$. Starting
from $R_1=\rho_1$, we obtain
\begin{equation}\label{up:root-error}
 0\le R_j-\rho_j\le\varepsilon^{-2}.
\end{equation}
Direct subtraction of the Charlier recurrence gives
\begin{equation}\label{up:Delta-recurrence}
 \Delta_1=1+1/R_1,\qquad
 \Delta_j=1+\frac1{R_j}+\frac{(j-1)\Delta_{j-1}}{R_jR_{j-1}}.
\end{equation}
It follows that
\begin{equation}\label{up:delta-bounds}
 1<\Delta_j\le2/\varepsilon,\qquad
 a_j\ge\frac r{r+1}\frac{1+\eta}{1-\eta}\ge1+\varepsilon.
\end{equation}

For the second differences, put $q_j:=(j-1)/(R_jR_{j-1})$, with $q_1:=0$.
Direct subtraction, followed by \eqref{up:root-bounds} and
\eqref{up:delta-bounds}, gives
\[
 0\le q_j-q_{j-1}\le\frac{36}{k\varepsilon},\qquad
 0\le\Delta_j-\Delta_{j-1}\le\frac{100}{k\varepsilon^3}.
\]
Indeed, the forcing in the latter recurrence is at most
$90/(k\varepsilon^2)$, while its multiplier is at most
$(1-\varepsilon)/(1+\varepsilon)$. Thus, for
$g_j:=\log a_j-\log a_{j+1}$ and $r/2\le j\le r$,
\begin{equation}\label{up:g-estimates}
 0<g_j=\log(1+1/j)+2\log(1+\Delta_j/R_{j+1})
 \le\frac{16}{k\varepsilon},\qquad
 |g_j-g_{j-1}|\le\frac{700}{k^2\varepsilon^3},
\end{equation}
where the second bound requires both indices to exceed $r/2$.

To lower-bound $g_r$, first suppose $\eta=\varepsilon$ and set
$m=\lfloor1/\varepsilon\rfloor$. On $r-m-2\le j\le r+1$,
\eqref{up:root-error} yields
\[
 R_j\le\frac{(k-1)(1+\varepsilon)}2+\frac4{\varepsilon^2}
       \le\frac{k(1+2\varepsilon)}2,\qquad
 \frac{4j}{k^2}\ge1-\varepsilon.
\]
Consequently
\begin{equation}\label{up:local-block}
 a_j\le1+8\varepsilon,\qquad q_j\ge1-6\varepsilon
 \quad(r-m-1\le j\le r).
\end{equation}
Iteration of \eqref{up:Delta-recurrence} gives
$\Delta_r\ge[1-(1-6\varepsilon)^m]/(6\varepsilon)\ge1/(8\varepsilon)$,
and hence $g_r\ge1/(4k\varepsilon)$. If $\eta\ge1/12$, instead
$g_r>2\log(1+1/(k-1))>2/k$. In either case,
\begin{equation}\label{up:center-g}
 g_r\ge\frac1{8k\varepsilon}.
\end{equation}

It remains to control convolution. Write
$h_j:=U_j/A_j=\sum_{i=0}^jd_iu_{j,i}$, where $u_{j,i}:=A_{j-i}/A_j$.
Then $u_{j,i}\le(1+\varepsilon)^{-i}$, $h_j\ge1$, and
\begin{equation}\label{up:h-lower}
 h_r\ge\frac1{200\varepsilon}.
\end{equation}
For $\varepsilon\ge1/64$ this is immediate. Otherwise,
\eqref{up:local-block} gives $u_{r,i}\ge(1+8\varepsilon)^{-i}>1/2$
for $2\le i\le\lfloor1/(16\varepsilon)\rfloor$; summing with $d_i\ge1/3$
proves \eqref{up:h-lower}.

For $i\le\lfloor r/3\rfloor$, let
$\delta_i^+:=\log(u_{r+1,i}/u_{r,i})$ and
$\delta_i^-:=\log(u_{r,i}/u_{r-1,i})$. Telescoping
\eqref{up:g-estimates} gives
\[
 0\le\delta_i^\pm\le\frac{16i}{k\varepsilon},\qquad
 |\delta_i^+-\delta_i^-|\le\frac{700i}{k^2\varepsilon^3}.
\]
The elementary inequality
$ue^\alpha-2u+ue^{-\beta}\le ue^\alpha(\alpha^2+|\alpha-\beta|)$
for $u>0$ and $\alpha,\beta\ge0$ therefore gives
\[
 u_{r+1,i}-2u_{r,i}+u_{r-1,i}
 \le(1+\varepsilon)^{-i}
 \left(\frac{256i^2}{k^2\varepsilon^2}
       +\frac{700i}{k^2\varepsilon^3}\right).
\]
Summing the geometric moments bounds these terms by
$2168/(k^2\varepsilon^5)$. The remaining tails, with absent terms set
to zero, contribute at most
\[
 2\sum_{i>\lfloor r/3\rfloor}(1+\varepsilon)^{-i}
 \le\frac4\varepsilon e^{-\varepsilon r/6}
 \le\frac{100}{k^2\varepsilon^5}.
\]
Since $0\le d_i\le1$, we conclude that
\[
 h_{r+1}-2h_r+h_{r-1}\le\frac{3000}{k^2\varepsilon^5},\qquad
 \log h_{r+1}-2\log h_r+\log h_{r-1}
 \le\frac{600000}{k^2\varepsilon^4},
\]
where the second inequality uses $\log x\le x-1$ and
\eqref{up:h-lower}. Combining with \eqref{up:center-g} proves
\[
 \log\frac{U_r^2}{U_{r-1}U_{r+1}}
 \ge\frac1{8k\varepsilon}-\frac{600000}{k^2\varepsilon^4}
 \ge\frac1{16k\varepsilon}.
\]
\end{proof}

\begin{lemma} \label{up:correction}
Under \eqref{up:near-domain}, if also $n\ge10^{20}$ and
$k\varepsilon^3\ge12\log n$, then $S/U_r\le n^{-3}$.
\end{lemma}
\begin{proof}
Set $c:=k/\sqrt n\in(2,4)$ and
\[
 \Phi:=\sqrt n\{2c\operatorname{arcosh}(c/2)-2\sqrt{c^2-4}\},\qquad
 h(c):=\log c-\operatorname{arcosh}(c/2).
\]
The positive function $h$ decreases and is bounded by $\log2$.
For $n-j\ge k$, \eqref{up:root-bounds} implies
$R_j\ge(n-j)\exp[-h((n-j)/\sqrt n)]$. Integral comparison gives
\begin{equation*}
 \log C_r(n)\ge\log\frac{(n-1)!}{(k-1)!}
 -h(c)-\sqrt n\int_c^\infty h(x)\,dx,\qquad
 \int_c^\infty h(x)\,dx=-ch(c)+c-\sqrt{c^2-4}.
\end{equation*}
Stirling's inequalities with remainder between $0$ and $1/(12m)$ give
\[
 2\log\frac{(n-1)!}{(k-1)!}-\log(r!)-\log(n!)
 \ge\sqrt n(2c-2c\log c)-\tfrac32\log n-19.
\]
Together with $U_r\ge C_r(n)^2/r!$, this proves
\begin{equation}\label{up:U-lower}
 \frac{U_r}{n!}\ge e^{-21}n^{-3/2}e^{-\Phi}\ge n^{-2}e^{-\Phi}.
\end{equation}

Under Poissonized Plancherel measure of mean $n$, the particles
$\lambda_i-i$ have kernel \cite[Theorem~2, Proposition~2.9 and Corollary~2.10]{BOO}
\[
 K_n(x,y):=\sum_{j\ge1}J_{x+j}(2\sqrt n)J_{y+j}(2\sqrt n).
\]
Set $\rho:=\exp(\operatorname{arcosh}(c/2))$. The Bessel coefficient
integral on $|z|=\rho$ gives
$|J_{k+\ell}(2\sqrt n)|\le e^{-\Phi/2}\rho^{-\ell}$ for $\ell\ge0$.
Thus, if $N_k$ counts particles at $x\ge k-1$,
\[
 \mathbb E N_k\le(1-\rho^{-2})^{-2}e^{-\Phi},\qquad
 \mathbb E[N_k(N_k-1)]\le(\mathbb E N_k)^2.
\]
The second inequality follows from the determinantal two-point
correlation and positivity of the kernel. Since
$\lambda_2\ge k+1$ implies $N_k\ge2$, conditioning on size $n$ and
using $\PP(\operatorname{Poisson}(n)=n)\ge1/(3\sqrt n)$ yields
\begin{equation}\label{up:S-upper}
 \frac S{n!}\le2\sqrt n\,(1-\rho^{-2})^{-4}e^{-2\Phi}.
\end{equation}
Finally set $v:=\sqrt{k^2-4n}/k$. The identity
$4n=(k+1)^2-(k-1)^2\eta^2$ and $k\varepsilon^3\ge100$ imply
$v\ge(99/100)\varepsilon$. Therefore
\[
 \Phi=2k(\operatorname{arctanh}v-v)\ge\tfrac35k\varepsilon^3,
 \qquad 1-\rho^{-2}=\frac{2v}{1+v}\ge\tfrac{99}{100}\varepsilon,
 \qquad \varepsilon^{-4}\le k^{4/3}<7n^{2/3}.
\]
Combining \eqref{up:U-lower} and \eqref{up:S-upper} gives
\[
 \frac S{U_r}\le28n^{19/6}e^{-(3/5)k\varepsilon^3}
 \le28n^{-121/30}\le n^{-3}.
\]
\end{proof}

\begin{lemma} \label{up:theorem}
For every integer $n\ge10^{300000}$ and $1<k<n$ satisfying
\begin{equation}\label{up:window}
 k\ge2\sqrt n+10n^{1/6}(\log n)^{2/3},
\end{equation}
one has $p_{n,k}^2>p_{n,k-1}p_{n,k+1}$.
\end{lemma}
\begin{proof}
If $k^2\ge16r$, apply Lemma~\ref{up:bridge-region}, since
$k>2\sqrt n>1000$. Otherwise
$k<4\sqrt n$ and $(k-1)^2-4r=(k+1)^2-4n>0$.
When $\eta\le1/12$, its definition and \eqref{up:window} give
\[
 k\varepsilon^3>
 \tfrac12\bigl((k-2\sqrt n)n^{-1/6}\bigr)^{3/2}
 \ge5\sqrt{10}\log n>\max(10^7,12\log n).
\]
When $\eta\ge1/12$, instead
$k\varepsilon^3>\sqrt n/864>\max(10^7,12\log n)$.
Thus Lemmas~\ref{up:near-curvature} and \ref{up:correction} apply, and
\eqref{up:transfer-basic} gives
\[
 \log\frac{P_{n,k}^2}{P_{n,k-1}P_{n,k+1}}
 \ge\frac1{16k\varepsilon}+2\log(1-n^{-3})
 \ge\frac1{16k\varepsilon}-4n^{-3}>0.
\]
Normalization by $n!$ preserves the inequality.
\end{proof}

\subsection{Case 3: the compact edge window}\label{compact:sec:compact}
\begin{lemma} \label{compact:thm:compact}
For every real $M\ge1$, put
\[
 N_c(M):=\left\lceil\exp\bigl(25000(M+2)^3\bigr)\right\rceil.
\]
If $n\ge N_c(M)$ and $|s|\le M$, then $\mathcal C_{n,k}>0$.
\end{lemma}

\begin{proof}
For a fixed $\alpha>0$, put $w:=\alpha+iy$ and introduce
the following interpolants on $\ell^2(\mathbb N_0)$:
\begin{equation}\label{compact:interpolants}
 \begin{aligned}
 b_n(z,x)&:=\frac1{2\pi}\int_{-\pi n^{1/6}}^{\pi n^{1/6}}
 e^{2\sqrt z\sinh(n^{-1/6}w)-2n^{1/3}w-xw}\,dy,\\
 (\mathcal H_n(z,x))_{ij}
 &:=n^{-1/6}b_n(z,x+(i+j+1)n^{-1/6}),\\
 \mathcal Q_n(z,x)&:=\det(I-\mathcal H_n(z,x)^2).
 \end{aligned}
\end{equation}
Their dependence on the fixed contour parameter $\alpha$ is suppressed
in the notation. Values away from the lattice may depend on $\alpha$;
within each comparison we use a single contour for all neighboring
indices. In the present proof, we restrict to $0<\alpha\le1$.
The coefficient circle of radius $e^{\alpha n^{-1/6}}$ in the Bessel
generating function gives, at the coordinate $t_k(n)$ from~\eqref{pre:coordinates},
\[
 (\mathcal H_n(z,t_k(n)))_{ij}=J_{k+i+j+1}(2\sqrt z),
 \qquad \mathcal Q_n(z,t_k(n))=Q_k(z).
\]
This is the exact discrete Bessel identity \cite{BOO}. The integer order,
$n$, and $t_k(n)$ stay fixed under $z$ differentiation; the square root
is principal on the positive-centered disks used below.

For $|z-n|\le\alpha n^{2/3}/8$ and $|\Ime x|\le\alpha/2$, the
unperturbed phase has real part at most
\[
 \frac25\alpha^3-\alpha\Ree x-\frac{4\alpha}{\pi^2}y^2
 +(\Ime x)y.
\]
The square-root perturbation contributes at most
$3\alpha^2/10+2\alpha|y|/5$. Completing the square gives
\begin{equation*}
 \Ree\{2\sqrt z\sinh(n^{-1/6}w)-2n^{1/3}w-xw\}
 \le2-\alpha\Ree x-\alpha y^2/5.
\end{equation*}
Here one uses $1-\cos u\ge2u^2/\pi^2$ for $|u|\le\pi$ and
$|\sinh(\alpha n^{-1/6}+iyn^{-1/6})|
 \le n^{-1/6}(6\alpha/5+8|y|/5)$.
Gaussian integration and a geometric sum yield
\begin{equation*}
 \|\mathcal H_n(z,x)\|_2^2
 \le\frac{e^4}{4\alpha^3}e^{-2\alpha\Ree x}.
\end{equation*}
These locally uniform majorants establish operator and determinant
holomorphy as explained after \eqref{pre:HSsquare}.
Henceforth choose $\alpha=(M+2)^{-1}$. For $\Ree x\ge-M-2$,
\begin{equation}\label{compact:eq:complexcompact}
 \|\mathcal H_n(z,x)\|_2^2<1000(M+2)^3,
 \qquad |\mathcal Q_n(z,x)|\le e^{1000(M+2)^3}.
\end{equation}
Freezing at any real $r\ge1$ and applying Cauchy's estimate gives
\begin{equation}\label{compact:eq:Qderiv}
 |Q_k^{(j)}(r)|\le j![8(M+2)]^je^{1000(M+2)^3}r^{-2j/3}
 \quad\text{when }t_k(r)\ge-M-2.
\end{equation}

For $|t_k(n)|\le M+1$ and $\log n\ge25000(M+2)^3$, we claim
\begin{equation}\label{compact:eq:depoiscompact}
 \left|a_{n,k}-Q_k(n)+\frac n2Q_k''(n)\right|
 \le e^{3000(M+2)^3}n^{-2/3}.
\end{equation}
Let $P$ be the degree-$511$ Taylor polynomial at $n$, and extract
coefficients on $|z|=n$, splitting at
$\theta_0=4\sqrt{\log n/n}$. On the inner arc,
$|z-n|\le4\sqrt{n\log n}$, whereas the analytic disk in
\eqref{compact:eq:complexcompact} has radius $n^{2/3}/[8(M+2)]$.
The Taylor remainder is therefore at most
\[
 2e^{1000(M+2)^3}
 \bigl[32(M+2)n^{-1/6}\sqrt{\log n}\bigr]^{512}
 \le e^{2000(M+2)^3}n^{-3}.
\]
The ratio in brackets is below $1/2$. The last estimate and all finite
coefficient sums below follow from $\sqrt{\log n}\le n^{1/42}$ and
\begin{align}
 1000(M+2)^3+512\log[128(M+2)]&+512\log512+256\log256\notag\\
 &+512\log\pi+\log512<2000(M+2)^3.
 \label{compact:eq:coeffenvelope}
\end{align}
Thus direct Cauchy estimates suffice on the inner arc.
The normalized coefficient kernel has absolute integral below two.
On the original outer arc, \eqref{pre:outerhaar} and the Gaussian-tail
bound give an error at most $2n^{-8/5}$. For $P$, the same calculation
term by term gives at most
$\tfrac12e^{2000(M+2)^3}n^{-8/5}$; the Gaussian moments are
$\pi^{j+1}\Gamma((j+1)/2)$, bounded using
\eqref{compact:eq:coeffenvelope}.

By~\eqref{pre:taylorcoeff}, the full polynomial integral is
$\sum_{j=0}^{511}b_j(n)Q_k^{(j)}(n)$. Using
\eqref{pre:charlierbounds} and~\eqref{compact:eq:Qderiv}, the terms
$j\ge3$ total at most $e^{2000(M+2)^3}n^{-2/3}$.
Together these bounds prove \eqref{compact:eq:depoiscompact}.

We compare $\mathcal H_n$ with the Airy Hankel operator $A_x$
from~\eqref{pre:TW-det}.
On $|\Ime x|\le\alpha/2$, the vertical Airy contour gives
\[
 |\Ai(x)|\le e\alpha^{-1/2}e^{-\alpha\Ree x},\qquad
 |\Ai'(x)|\le5\alpha^{-1}e^{-\alpha\Ree x}.
\]
At $z=n$, both the Bessel and Airy phases have real part at most
$1-\alpha\Ree x-\alpha y^2/5$, and their difference has modulus at
most $n^{-1/3}|w|^5$. Integrating against this sharper Gaussian
majorant and bounding the omitted
Airy tails by $1\le n^{-1/3}y^2/\pi^2$, gives
\begin{equation}\label{compact:eq:BesselAiry}
 |b_n(n,x)-\Ai(x)|\le4000n^{-1/3}\alpha^{-3}e^{-\alpha\Ree x}.
\end{equation}
For the constant, $|w|^5\le16(1+|y|^5)$ bounds the finite-contour
contribution by $2032n^{-1/3}\alpha^{-3}e^{-\alpha\Ree x}$;
the tail contributes less than $n^{-1/3}\alpha^{-3}e^{-\alpha\Ree x}$.

Embed $\ell^2(\mathbb N_0)$ isometrically into $L^2(0,\infty)$ by
$e_i\mapsto n^{1/12}\mathbf1_{[in^{-1/6},(i+1)n^{-1/6})}$.
The embedded operator $C_{n,x}$ has kernel
$b_n(n,x+(i+j+1)n^{-1/6})$ on the corresponding pair of cells, vanishes
on the orthogonal complement of the step functions, and satisfies
$\det(I-C_{n,x}^2)=\mathcal Q_n(n,x)$.
The cell midpoint differs from $u+v$ by at most $n^{-1/6}$.
Consequently \eqref{compact:eq:BesselAiry} and the Airy derivative bound
imply, throughout the same complex strip with $\Ree x\ge-M-2$,
\[
 \begin{gathered}
 \|C_{n,x}-A_x\|_2\le20000n^{-1/6}(M+2)^4,\qquad 
 \|C_{n,x}\|_2\le32(M+2)^{3/2},\\
 \|A_x\|_2\le4(M+2)^{3/2},\qquad 
 \|C_{n,x}^2-A_x^2\|_1\le10^6n^{-1/6}(M+2)^6.
 \end{gathered}
\]
The determinant perturbation bound \eqref{pre:HSsquare} gives
\[
 |\mathcal Q_n(n,x)-F(x)|
 \le10^6(M+2)^6n^{-1/6}e^{1+1016(M+2)^3}
 \le e^{1100(M+2)^3}n^{-1/6}.
\]

Define the corrected distribution and its centered mass model by
\begin{equation}\label{compact:mass-model}
 \begin{aligned}
 S_n(x)&:=\mathcal Q_n(n,x)-\frac n2\partial_z^2\mathcal Q_n(n,x),\\
 B_n(x)&:=n^{1/6}\{S_n(x+\tfrac12n^{-1/6})-S_n(x-\tfrac12n^{-1/6})\}.
 \end{aligned}
\end{equation}
Cauchy's estimate on the same $z$ disk bounds the correction by
$64(M+2)^2e^{1000(M+2)^3}n^{-1/3}$, so
$|S_n-F|\le e^{1200(M+2)^3}n^{-1/6}$.
On real $|x|\le M+1$, disks of radius $1/[4(M+2)]$ remain in the strip.
Cauchy's estimates through order four, followed by averaging $S_n'$,
therefore give, for $|x|\le M+n^{-1/6}$ and $f=F'$,
\begin{equation}\label{compact:eq:smoothmodel}
 |B_n^{(j)}(x)-f^{(j)}(x)|\le e^{1300(M+2)^3}n^{-1/6},\qquad
 |f^{(j)}(x)|\le e^{1300(M+2)^3}\quad(j=0,1,2).
\end{equation}
For the stated derivatives of $f$, the same Cauchy estimate uses
$|F|\le e^{16(M+2)^3}$ and
$6[4(M+2)]^3e^{16(M+2)^3}\le e^{1300(M+2)^3}$.
Conjugation symmetry makes the model real on the real axis.

Lemma~\ref{pre:TWcompact} gives, on $[-M-2,M+2]$,
\begin{equation*}
 f(x)\ge e^{-10(M+4)^3},\qquad
 -(\log f)''(x)\ge e^{-20(M+4)^2}.
\end{equation*}
Since $n^{-1/6}\le e^{-25000(M+2)^3/6}$,
\eqref{compact:eq:smoothmodel} implies $B_n\ge f/2$ and, by subtracting
the two logarithmic second derivatives,
\[
 |(\log B_n)''-(\log f)''|
 \le24e^{3900(M+2)^3+30(M+4)^3}n^{-1/6}
 \le e^{4100(M+2)^3}n^{-1/6}
 <\tfrac12 e^{-20(M+4)^2}.
\]
Integrating against the triangular kernel yields
\begin{equation}\label{compact:eq:modelcurv}
 2\log B_n(s)-\log B_n(s-n^{-1/6})-\log B_n(s+n^{-1/6})
 \ge\tfrac12e^{-20(M+4)^2}n^{-1/3}.
\end{equation}
Using the coordinates in~\eqref{pre:coordinates}, at $j=k-1,k,k+1$ the exact lattice identity and
\eqref{compact:eq:depoiscompact} give
\[
 \begin{aligned}
 |p_{n,j}-n^{-1/6}B_n(s_{n,j})|
 &\le2e^{3000(M+2)^3}n^{-2/3},\\
 \left|\frac{p_{n,j}}{n^{-1/6}B_n(s_{n,j})}-1\right|
 &\le4e^{3000(M+2)^3+10(M+4)^3}n^{-1/2}<\tfrac12.
 \end{aligned}
\]
All required distribution-function coordinates lie in $[-M-1,M+1]$.
For fixed $M$, the second-order correction thus gives relative mass
error $O(n^{-1/2})$, smaller than the $n^{-1/3}$ model curvature in
\eqref{compact:eq:modelcurv}; the displayed constants retain the
dependence on $M$.
Thus the total logarithmic error in the three-point curvature is at most
\[
 32e^{3000(M+2)^3+10(M+4)^3}n^{-1/2}
 \le e^{4000(M+2)^3}n^{-1/2}
 <\tfrac14e^{-20(M+4)^2}n^{-1/3}.
\]
This is smaller than \eqref{compact:eq:modelcurv}, completing the proof of
Lemma~\ref{compact:thm:compact}.
\end{proof}

\subsection{Case 4: the right logarithmic window}\label{right:section}
We retain the right-tail exponential in both determinant approximation
and coefficient extraction, using the coordinates in~\eqref{pre:coordinates}.

\begin{lemma} \label{right:TW-tail}
For every real $x\ge22$, the GUE Tracy--Widom density $f=F'$ satisfies
\begin{equation}\label{right:TW-bounds}
 f(x)\ge\frac{e^{-4x^{3/2}/3}}{128x},\qquad
 -(\log f(x))''\ge\frac1{8\sqrt x}.
\end{equation}
\end{lemma}

\begin{proof}
The Airy diagonal $K$ in~\eqref{pre:airy-diagonal} is positive.
Its integral and differential expressions agree by differentiation
and decay at infinity. Since
$\Ai>0$ and $\Ai'<0$ on the positive axis, it also gives
\begin{equation}\label{right:airy-decay}
 -\Ai'(x)/\Ai(x)\ge\sqrt x,\qquad
 \Ai(t)\le\Ai(x)e^{-\sqrt x(t-x)}\quad(t\ge x>0).
\end{equation}
Normalize the two Airy functions by
\[
 A_0(x):=2\sqrt\pi x^{1/4}e^{2x^{3/2}/3}\Ai(x),\qquad
 A_1(x):=-2\sqrt\pi x^{-1/4}e^{2x^{3/2}/3}\Ai'(x).
\]
The signed first-neglected-term bounds for the Airy expansions
\cite[(9.7.2), (9.7.5)--(9.7.6) and \S9.7(iii)]{DLMF} give
\begin{equation*}
 \begin{aligned}
 A_0(x)&\ge A_L(x^{-3/2}),&
 A_L(z)&:=1-\frac5{48}z+\frac{385}{4608}z^2
                    -\frac{85085}{663552}z^3,\\
 A_1(x)&\le A_U(x^{-3/2}),&
 A_U(z)&:=1+\frac7{48}z-\frac{455}{4608}z^2
                    +\frac{5005}{18432}z^3.
 \end{aligned}
\end{equation*}
Here the $z^3$ coefficient of the standard upper bound for $A_1$ is
$95095/663552\le5005/18432$, so the displayed weakening preserves that
bound. The same signed remainders give, for $x^{-3/2}\le1/100$,
\begin{equation*}
 \frac{99}{100}<1-\frac5{48}x^{-3/2}\le A_0(x)\le1,
 \quad\frac{x^{-3/2}}5\le A_1(x)-A_0(x)\le\frac{x^{-3/2}}4,
 \quad\frac{99}{100}<A_1(x)<\frac{101}{100}.
\end{equation*}
Consequently
\begin{equation}\label{right:airy-ratios}
 K(x)\ge\frac{e^{-4x^{3/2}/3}}{16\pi x},\qquad
 \frac{\Ai(x)^2}{K(x)}\le4\sqrt x,\qquad
 -\frac{\Ai'(x)}{\Ai(x)}\le2\sqrt x.
\end{equation}

For the curvature, define the polynomial
\[
 \Phi(A,B,z):=zA^4-2AB(B^2-A^2)-\tfrac14(B^2-A^2)^2.
\]
On $A,B\in[99/100,101/100]$ and $z\ge0$, its $A$ derivative is
at least $370119/100000>0$, and its $B$ derivative is negative.
Both $A_L(z)$ and $A_U(z)$ remain in that box for $0\le z\le1/100$.
Direct multiplication of the displayed cubics gives the rational
certificate
\[
 \Phi(A_L(z),A_U(z),z)=z^2\sum_{j=0}^{11}c_jz^j,
 \quad c_0=\frac3{16},\quad c_1=-\frac{170621}{165888},
 \quad |c_j|\le2\ (2\le j\le11).
\]
Thus
\[
 z^{-2}\Phi(A_L(z),A_U(z),z)
 \ge\frac3{16}-\frac{170621}{16588800}-\frac2{9900}
 >\frac{17}{100}.
\]
Since $K'=-\Ai^2$, $K''=-2\Ai\Ai'$, and
\[
 \sqrt x[-(\log K)'']
 =\frac{x^{-3/2}A_0^4-2A_0A_1(A_1^2-A_0^2)}{(A_1^2-A_0^2)^2},
\]
this proves
\begin{equation*}
 -(\log K(x))''\ge\frac1{4\sqrt x}\qquad(x\ge22).
\end{equation*}

We now use the Hastings--McLeod solution and the identities
in~\eqref{pre:TW-def}. The ratio $r:=q/\Ai$ satisfies
\begin{equation*}
 (\Ai^2r')'=2\Ai^4r^3,\qquad
 r'(x)=-2\Ai(x)^{-2}\int_x^\infty\Ai(t)^4r(t)^3\,dt.
\end{equation*}
The boundary term vanishes because $q\sim\Ai$, the equation makes
$q''$ integrable, and hence $q'\to0$. On its maximal positive tail,
$r'\le0$ and $r\ge1$, excluding a finite zero of $q$.
Where $r\le2$, \eqref{right:airy-decay} gives
\[
 0\le-r'(x)\le4\Ai(x)^2/\sqrt x,\qquad
 0\le r(x)-1\le\epsilon(x):=2\Ai(x)^2/x<1.
\]
A first crossing of $r=2$ is impossible, so these bounds hold for all
$x\ge22$. Differentiation and integration respectively give
\begin{equation*}
 \begin{gathered}
 q=\Ai(1+\delta),\qquad -q'=-\Ai'(1+\delta_v),\qquad
 u=K(1+\eta),\\
 0\le\delta\le\epsilon,\quad0\le\delta_v\le3\epsilon,
 \quad0\le\eta\le3\epsilon.
 \end{gathered}
\end{equation*}
Using $-(\log f)''=q^2+q^4/u^2+2qq'/u$, these relative bounds and
\eqref{right:airy-ratios} imply
\[
 \left|-(\log u)''+(\log K)''\right|
 \le21\epsilon\frac{\Ai^4}{K^2}
       +20\epsilon\frac{-\Ai\Ai'}K
 \le496\epsilon x=992\Ai(x)^2<\frac1{8\sqrt x}.
\]
For example, the two relative quotients differ from one by at most
$21\epsilon$ and $10\epsilon$, respectively. The final inequality
follows from $A_0\le1$ and $(4/3)x^{3/2}>17$.
This proves the curvature assertion. Finally, $q\le2\Ai$ gives
$\int_x^\infty(t-x)q(t)^2dt\le\Ai(x)^2/x<1/2$, whence
$F\ge1/2$ and $f\ge K/2$. Equation~\eqref{right:airy-ratios} proves
the stated density bound.
\end{proof}

\begin{lemma}
\label{right:weighted-determinant}
Fix $n$ and $s\ge22$ with $n^{-1/6}s^2\le1$.
Use the interpolants in~\eqref{compact:interpolants} with the fixed
contour parameter $\alpha:=\sqrt s$. These are holomorphic on the product of disks
\begin{equation*}
 |x-s|\le2/\sqrt s,\qquad |z-n|\le n^{2/3}/(100\sqrt s),
\end{equation*}
with the positive-centered square-root branch, and satisfy
\begin{equation}\label{right:lattice-equality}
 \mathcal Q_n(z,t_j(n))=Q_j(z).
\end{equation}
On these disks,
$|1-\mathcal Q_n(z,x)|\le e^{10}s^{-3/2}e^{-4s^{3/2}/3}$.
Moreover, the corrected distribution $S_n$
defined by~\eqref{compact:mass-model}, with this choice of contour,
satisfies
\begin{equation}\label{right:S-error}
 |S_n(x)-F(x)|\le e^{31}n^{-1/6}s^{-1}e^{-4s^{3/2}/3}.
\end{equation}
\end{lemma}
\begin{proof}
The Bessel coefficient circle of radius $e^{\sqrt s n^{-1/6}}$
identifies the entries at $x=t_j(n)$ with
$J_{j+u+v+1}(2\sqrt z)$. The discrete Bessel gap identity
\cite[Theorem~2, Proposition~2.9 and Corollary~2.10]{BOO} gives
\eqref{right:lattice-equality}, extended to the complex disk by
analytic continuation. The integer order remains fixed under
$z$ differentiation.

At $z=n$, the phase differs from $w^3/3-xw$ by at most
$n^{-1/3}|w|^5$. Its constant real part is at most $-2s^{3/2}/3+3$,
and its remaining part is at most
$-4\sqrt s\,y^2/\pi^2+2|y|/\sqrt s$.
The $z$ perturbation adds at most $0.024+0.032|y|/\sqrt s$.
Completing the square, using
$2.032|y|/\sqrt s\le\sqrt s\,y^2/5+5.16128/s^{3/2}$, yields
\begin{equation}\label{right:phase-envelope}
 \Ree\{2\sqrt z\sinh(n^{-1/6}w)-2n^{1/3}w-xw\}
 \le-2s^{3/2}/3+4-\sqrt s\,y^2/5.
\end{equation}
The Airy phase satisfies the same envelope. Gaussian integration and
geometric summation give
\begin{equation*}
 |(\mathcal H_n(z,x))_{uv}|
 \le e^4n^{-1/6}s^{-1/4}e^{-2s^{3/2}/3-(u+v+1)\sqrt s n^{-1/6}},
 \qquad
 \|\mathcal H_n(z,x)\|_2^2\le\frac{e^8}{4}s^{-3/2}e^{-4s^{3/2}/3}<\frac12.
\end{equation*}
The locally uniform envelope gives Hilbert--Schmidt holomorphy and
hence determinant holomorphy. The exterior-power bound after
\eqref{pre:HSsquare} proves the asserted estimate for $1-\mathcal Q_n$.
The Airy Hankel operator satisfies, on the $x$ disk,
\begin{equation*}
 \|A_x\|_2\le\tfrac12e^3s^{-3/4}e^{-2s^{3/2}/3},\qquad
 |1-F(x)|\le e^{10}s^{-3/2}e^{-4s^{3/2}/3}.
\end{equation*}

The phase comparison in \eqref{right:phase-envelope} also gives
\begin{equation}\label{right:b-Ai-error}
 |b_n(n,x+u)-\Ai(x+u)|
 \le e^{13}n^{-1/3}s^{9/4}e^{-2s^{3/2}/3-\sqrt s\,u}
 \qquad(u\ge0).
\end{equation}
Indeed $|\sqrt s+iy|^5\le16(s^{5/2}+|y|^5)$,
$\int e^{-\sqrt s y^2/5}dy<4s^{-1/4}$, and
$\int|y|^5e^{-\sqrt s y^2/5}dy=250s^{-3/2}$ give a finite-contour
constant below $e^{12}$. The omitted Airy tails contribute below $e^5$
with the same factor $n^{-1/3}s^{9/4}$.
The same contour gives
$|\Ai'(x+u)|\le e^4s^{1/4}e^{-2s^{3/2}/3-\sqrt s\,u}$.

Use the isometric step-function embedding from the compact window,
so the embedded operator $C_{n,x}$ satisfies
$\det(I-C_{n,x}^2)=\mathcal Q_n(n,x)$ and
$\|C_{n,x}\|_2=\|\mathcal H_n(n,x)\|_2$.
The midpoint error is at most $n^{-1/6}$; hence
\eqref{right:b-Ai-error} and $n^{-1/6}s^2\le1$ give
\[
 \begin{gathered}
 \|C_{n,x}-A_x\|_2\le e^{15}n^{-1/6}s^{-1/4}e^{-2s^{3/2}/3},\\
 \|C_{n,x}\|_2+\|A_x\|_2\le e^5s^{-3/4}e^{-2s^{3/2}/3},\\
 \|C_{n,x}^2-A_x^2\|_1\le e^{20}n^{-1/6}s^{-1}e^{-4s^{3/2}/3}.
 \end{gathered}
\]
Applying \eqref{pre:HSsquare}, whose exponential multiplier is
less than $e^3$, gives
$|\mathcal Q_n(n,x)-F(x)|\le e^{30}n^{-1/6}s^{-1}e^{-4s^{3/2}/3}$.
Cauchy's estimate applied to $1-\mathcal Q_n$ on the $z$ disk gives
\[
 \frac n2|\partial_z^2\mathcal Q_n(n,x)|
 \le10^4e^{10}n^{-1/3}s^{-1/2}e^{-4s^{3/2}/3}
 \le e^{20}n^{-1/3}s^{-1/2}e^{-4s^{3/2}/3}.
\]
Since $n^{-1/6}\sqrt s\le1$, this proves \eqref{right:S-error}.
\end{proof}

\begin{lemma} \label{right:smooth-curvature}
Under Lemma~\ref{right:weighted-determinant}, assume $s\ge23$ and
$n^{-1/6}\sqrt s\le e^{-100}$. Let $B_n$ be the centered mass model
in~\eqref{compact:mass-model}, with the same fixed contour as in that lemma.
For real $|x-s|\le n^{-1/6}$, $B_n(x)>0$ and
\begin{equation*}
 |(\log B_n(x))''-(\log f(x))''|\le e^{200}n^{-1/6}s^{3/2}.
\end{equation*}
If also $e^{200}n^{-1/6}s^2\le1/32$, then
\begin{equation}\label{right:B-discrete-margin}
 2\log B_n(s)-\log B_n(s-n^{-1/6})-\log B_n(s+n^{-1/6})
 \ge\frac{n^{-1/3}}{32\sqrt s}.
\end{equation}
\end{lemma}
\begin{proof}
Cauchy's estimates on radius-$s^{-1/2}$ disks centered within
$s^{-1/2}$ of $s$ give
\begin{align}
 |S_n^{(j)}(x)-F^{(j)}(x)|
 &\le e^{33}n^{-1/6}s^{(j-2)/2}e^{-4s^{3/2}/3}
 &&(0\le j\le3),\notag\\
 |F^{(r)}(x)|&\le e^{15}s^{(r-3)/2}e^{-4s^{3/2}/3}
 &&(1\le r\le4).\label{right:Cauchy-F}
\end{align}
For the second estimate, apply Cauchy's formula to $1-F$ so that the
exponential weight is retained. Averaging $S_n'$ in \eqref{compact:mass-model}
yields
\begin{equation*}
 |B_n^{(j)}(x)-f^{(j)}(x)|
 \le e^{35}n^{-1/6}s^{(j-1)/2}e^{-4s^{3/2}/3}
 \quad(j=0,1,2;\ |x-s|\le n^{-1/6}).
\end{equation*}
All averaging intervals remain in the Cauchy domain.
Lemma~\ref{right:TW-tail} gives on this real neighborhood
\[
 f(x)\ge e^{-10}s^{-1}e^{-4s^{3/2}/3},\qquad
 -(\log f(x))''\ge\frac1{16\sqrt s}.
\]
Together with \eqref{right:Cauchy-F}, this bounds $|f'/f|$ and
$|f''/f|$ by $e^{25}\sqrt s$ and $e^{25}s$.
For the relative error $v=(B_n-f)/f$, quotient differentiation gives
\[
 |v|\le e^{45}n^{-1/6}\sqrt s,\qquad
 |v'|\le e^{71}n^{-1/6}s,\qquad
 |v''|\le e^{98}n^{-1/6}s^{3/2}.
\]
Thus $|v|\le1/2$, and
$|[\log(1+v)]''|\le2|v''|+4|v'|^2\le e^{200}n^{-1/6}s^{3/2}$.
Under the final hypothesis this leaves
$-(\log B_n)''\ge1/(32\sqrt s)$. Integrating against the triangular
kernel proves \eqref{right:B-discrete-margin}.
\end{proof}

\begin{lemma} \label{right:depoisson}
Suppose $\log n\ge1000$ and $j$ is an integer with
$22\le t_j(n)\le12(\log n)^{2/3}$. Then
\begin{equation}\label{right:CDF-error}
 \left|a_{n,j}-Q_j(n)+\frac n2Q_j''(n)\right|
 \le e^{100}\sqrt{t_j(n)}\,e^{-4t_j(n)^{3/2}/3}n^{-2/3}+n^{-998}.
\end{equation}
\end{lemma}
\begin{proof}
Lemma~\ref{right:weighted-determinant} applies at $s=t_j(n)$.
Cauchy's estimate on its $z$ disk gives
\begin{equation}\label{right:CDF-derivative}
 \frac{|Q_j^{(r)}(n)|}{r!}
 \le e^{10}t_j(n)^{-3/2}e^{-4t_j(n)^{3/2}/3}
       [100\sqrt{t_j(n)}\,n^{-2/3}]^r\qquad(r\ge1).
\end{equation}
Let $P$ be the degree-nine Taylor polynomial at $n$, and split the
coefficient integral on $|z|=n$ at $\theta_0=100\sqrt{\log n/n}$.
The inner-arc displacement divided by the analytic radius is at most
$10000\sqrt{t_j(n)}\,n^{-1/6}\sqrt{\log n}<1/2$, and
\begin{equation*}
 4[10000\sqrt{t_j(n)}\,n^{-1/6}\sqrt{\log n}]^{10}
 \le t_j(n)^2n^{-2/3}.
\end{equation*}
Indeed the ratio is bounded by
$4\cdot10^{40}\cdot4096(\log n)^7/n<1$ for $\log n\ge1000$.
Since the normalized coefficient kernel has absolute integral below two,
the inner Taylor remainder contributes at most
$e^{10}\sqrt{t_j(n)}e^{-4t_j(n)^{3/2}/3}n^{-2/3}$.

The integer-index Haar bound~\eqref{pre:haarbound} makes the original outer arc less than $n^{-1000}$, since
$10000/\pi^2>1000$.
For the polynomial outer arc, \eqref{right:CDF-derivative} bounds its
$r$th term by $|\sqrt n\theta|^r$ for $1\le r\le9$; the constant
term is at most one because $0\le Q_j(n)\le1$.
Splitting the Gaussian into two factors and using
\[
 \sum_{r=0}^9\int_{\mathbb R}|v|^re^{-v^2/\pi^2}\,dv
 =\sum_{r=0}^9\pi^{r+1}\Gamma((r+1)/2)<10^{10}
\]
gives an error below $n^{-999}$. Both outer arcs therefore total
at most $n^{-998}$.

Apply the exact coefficient identity~\eqref{pre:taylorcoeff} and
the coefficient bounds~\eqref{pre:charlierbounds}.
For each $3\le r\le9$, \eqref{right:CDF-derivative} bounds
$|b_r(n)Q_j^{(r)}(n)|$ by
\[
 r!100^re^{10}\sqrt{t_j(n)}\,e^{-4t_j(n)^{3/2}/3}n^{-2/3}.
\]
For even $r$ this uses $\sqrt{t_j(n)}n^{-1/6}\le1$. Higher odd
indices have an additional factor $n^{-1/2}$; for $r=3$ the ratio to
the displayed bound is $t_j(n)^{-1/2}n^{-1/3}\le1$.
Since $\sum_{r=3}^9r!100^r<10^{30}$, adding these terms and the
Taylor remainder proves \eqref{right:CDF-error}.
\end{proof}

\begin{lemma} \label{right:theorem}
If $n\ge10^{300000}$ and
\begin{equation}\label{right:window}
 23\le s_{n,k}\le11(\log n)^{2/3},
\end{equation}
then $p_{n,k}^2>p_{n,k-1}p_{n,k+1}$.
\end{lemma}
\begin{proof}
Write $s=s_{n,k}$. All four distribution-function endpoints lie within
$3n^{-1/6}/2$ of $s$, hence in $[22,12(\log n)^{2/3}]$.
Since $\log n>600000$ and $\sqrt s\le4(\log n)^{1/3}$,
\begin{equation*}
 n^{-1/6}\sqrt s\le e^{-100},\qquad n^{-1/6}s^2\le1,
 \qquad e^{200}n^{-1/6}s^2\le1/32.
\end{equation*}
For the strongest inequality it suffices that
$200+\log256+\tfrac43\log\log n-\tfrac16\log n\le-\log32$;
the left side decreases and is already smaller at $\log n=600000$.
Thus the smooth model has curvature margin $n^{-1/3}/(32\sqrt s)$.

Use the same contour through $\sqrt s$ for all three modeled masses.
At integer endpoints, \eqref{right:lattice-equality} identifies the
values and $z$ derivatives independently of the contour used to prove
the endpoint coefficient bounds. Neighboring square-root factors
change by less than two and exponential weights by less than $e$.
Taking successive differences in \eqref{right:CDF-error} gives
\begin{equation*}
 |p_{n,j}-n^{-1/6}B_n(s_{n,j})|
 \le e^{104}\sqrt s\,e^{-4s^{3/2}/3}n^{-2/3}+2n^{-998}
 \quad(j=k-1,k,k+1).
\end{equation*}
Here $B_n\ge f/2\ge\tfrac12e^{-10}s^{-1}e^{-4s^{3/2}/3}$.
Also $f\ge n^{-60}$ on the entire endpoint range: use
$4x^{3/2}/3<56\log n$ and $128x<n$ in \eqref{right:TW-bounds}.
Consequently
\[
 \left|\frac{p_{n,j}}{n^{-1/6}B_n(s_{n,j})}-1\right|
 \le e^{120}s^{3/2}n^{-1/2}+n^{-930}<\tfrac12.
\]
The total logarithmic error in the three-point curvature is therefore
at most
\[
 e^{130}s^{3/2}n^{-1/2}+n^{-929}
 <\frac{n^{-1/3}}{64\sqrt s}.
\]
Each summand is below $n^{-1/3}/(128\sqrt s)$: the first follows from
$128e^{130}n^{-1/6}s^2\le4e^{-70}<1$; the second from
$\log512+\tfrac13\log\log n-(929-1/3)\log n<0$.
This error is smaller than the model margin, proving strict curvature.
\end{proof}

The upper boundary of \eqref{right:window} overlaps Lemma~\ref{up:theorem},
whose lower boundary is $10(\log n)^{2/3}-\tfrac12n^{-1/6}$.
The lower boundary $s_{n,k}=23$ of~\eqref{right:window} lies inside
the compact window $|s_{n,k}|\le24$.

\subsection{Case 5: the drifting left region}\label{left:sec:transfer}
For this region use the saddle $t=n/k+k/4$, its gap $\Delta=2t-k$, and
\[
 Q(c):=\log\frac{1+c^2/4}{c^2/2}
       +\frac{(1-c^2/4)^2}{(c^2/2)(1+c^2/4)}.
\]
\begin{lemma} \label{left:thm:left}
If $t\ge10^{60}$ and
$2\cdot10^{20}t^{1/3}\le\Delta\le2\cdot10^{-13}t$, then
\begin{equation}\label{left:eq:leftresult}
 \mathcal C_{n,k}\ge Q(k/\sqrt n)
 -3\cdot10^{57}(\Delta^{-2}+t^{-2})-10^{30}/t
 \ge\Delta/(4t)>0.
\end{equation}
In particular, strict log-concavity holds when
\begin{equation}\label{left:eq:leftwindow}
 n\ge10^{120},\qquad(2-10^{-13})\sqrt n\le k<2\sqrt n,
 \qquad T\ge T_0:=2\cdot10^{20}.
\end{equation}
\end{lemma}

 \begin{proof}
At index $j$, write $u=n/j+j/4$ and $d=2u-j$. We first work on
\begin{equation}\label{left:eq:widedomain}
 u\ge10^{57},\qquad\frac95\,10^{20}u^{1/3}\le d\le3\cdot10^{-13}u.
\end{equation}
For $|q-j|\le\lfloor d/4\rfloor$, the gap lies between $3d/4$ and
$5d/4$, and $u\le q\le2u$. Even with one additional recurrence index,
\[
 b_q^2\le1.26(3\cdot10^{-13})<10^{-12},\qquad
 \Lambda_q^2\ge\frac{0.74^3}{2}\left(\frac95\right)^3 10^{60}
 >1.18\cdot10^{60}.
\]
Also $d>100$, $d^3/u>64\cdot10^{26}$ and $\Lambda_q^3>10^{27}$.
Thus all reflection, exact-mass and derivative estimates apply.

Use $g_j$ and the coefficient formula from~\eqref{pre:massgen}.
On the coefficient circle, set
$\theta=y/\sqrt{ju/2}$ and write the square root of the coefficient
variable as $z=u e^{i\theta/2}$; here
$u^2/2\le ju/2\le u^2$. The exact action satisfies
\begin{equation*}
 V_j(z)-V_j(u)-in\theta
 =2ju(e^{i\theta/2}-1-i\theta/2)
 =-\frac{y^2}{2}-\frac{iy^3}{12\sqrt{ju/2}}+q_V,
 \qquad |q_V|\le\frac{y^4}{48u^2}.
\end{equation*}
For $|y|\le100\sqrt{\log u}$, the central arc stays in the half disk:
$100\sqrt{\log u}\le10^{-4}d/4$. Since
$|z-u|\le|y|$ and $|z-u-iu\theta/2|\le y^2/(4u)$,
\eqref{der:eq:Kconstants} gives
\begin{equation*}
 \mathscr R_j(z)-\mathscr R_j(u)=i\lambda y+q_R,\qquad
 \lambda=\frac{u\mathscr R_j'(u)}{2\sqrt{ju/2}}\in\mathbb R,
 \quad |\lambda|\le10^{28}/d,\quad
 |q_R|\le2\cdot10^{32}y^2/d^2.
\end{equation*}
Throughout this arc, $|q_V+q_R|\le y^2/8$.
The remaining phase
$\psi=\lambda y-y^3/(12\sqrt{ju/2})$ is real and odd.
Use $|e^{i\psi}-1-i\psi|\le\psi^2/2$ and
$|e^q-1|\le|q|e^{\max(\Ree q,0)}$; the integral of $i\psi$ against
the symmetric Gaussian vanishes. The standard Gaussian moments,
including
\[
 (2\pi)^{-1/2}\int y^2e^{-3y^2/8}\,dy<2,
 \qquad(2\pi)^{-1/2}\int y^4e^{-3y^2/8}\,dy<7,
\]
therefore bound the normalized central error by
\begin{equation}\label{left:eq:centralleft}
 (10^{56}+4\cdot10^{32})d^{-2}+2u^{-2}.
\end{equation}
This cancellation is essential; no pointwise smallness of $\psi$ is needed.

For the remaining arcs, Lemma~\ref{lem:massfactor} gives
$D_j(u)/g_j(u^2)\le4\max(j/d^2,1)\le8u^{1/3}$.
The angular integral consequently yields
\begin{equation}\label{left:eq:Haarleft}
 \frac{|g_j(u^2e^{i\theta})|}{g_j(u^2)}
 \le16u^{1/3}\frac{D_j(u\cos(\theta/2))}{D_j(u)}.
\end{equation}
For $\theta^2\le10^{-4}d/u$, the reduced radius stays within
$10^{-4}d/8$ of $u$. On this segment,
\[
 (\log D_j)'\ge u-(75\cdot10^{26}+1)/d\ge u/2,
\]
so its ratio in \eqref{left:eq:Haarleft} is at most
$e^{-u^2\theta^2/40}\le e^{-y^2/40}$.
The normalized intermediate tail is bounded by
$112u^{1/3}u^{-125}\le u^{-100}$.
Beyond this angular range, monotonicity of $D_j$ bounds the entire
outer arc by $48u^{4/3}e^{-ud/400000}\le u^{-100}$.
The omitted Gaussian tail is also at most $u^{-100}$.
All three comparisons follow from \eqref{left:eq:widedomain} and
$\log u\le u^{1/3}$; in particular, the central arc ends before the
intermediate one.
Combining these bounds with \eqref{left:eq:centralleft} proves
\begin{equation*}
 p_{n,j}=\frac{n!g_j(u^2)}{u^{2n}\sqrt{\pi ju}}(1+\varepsilon_j),
 \qquad |\varepsilon_j|\le2\cdot10^{56}(d^{-2}+u^{-2})<\tfrac12.
\end{equation*}
Define $\eta_j:=\log(1+\varepsilon_j)$. Then
\begin{equation}\label{left:eq:lefteta}
 |\eta_j|\le4\cdot10^{56}(d^{-2}+u^{-2}).
\end{equation}
The amplitude $g_j(u^2)$ remains exact, so no additive determinant
normalization has been lost.

Write $\tau_j:=n/j+j/4$ for the neighboring saddles, so $\tau_k=t$.
At the central saddle, $n=t^2-\Delta^2/4$ and
\[
 \tau_{k+1}-t=-\frac{2\Delta-1}{4(k+1)},\qquad
 \tau_{k-1}-t=\frac{2\Delta+1}{4(k-1)}.
\]
Hence $|\tau_j-t|\le\Delta/t$, each gap changes by less than $2$, and
both neighboring saddles satisfy \eqref{left:eq:widedomain}, with
$0.99t\le \tau_j\le1.01t$ and
$0.99\Delta\le2\tau_j-j\le1.01\Delta$.
All indices $k-2,\ldots,k+1$ lie in the required reflection windows.
The total contribution of \eqref{left:eq:lefteta} to curvature is at most
\begin{equation}\label{left:eq:etaCurv}
 24\cdot10^{56}(\Delta^{-2}+t^{-2}).
\end{equation}

At the common radius $t$, the Toeplitz determinant identity
\cite[Corollary~2.8]{CT}, in our indexing convention, reads
\[
\frac{D_{k-1}(t)D_{k+1}(t)}{D_k(t)^2}
=1-\beta_k(t)^2.
\]
Taking logarithms gives
\[
2\log D_k(t)-\log D_{k-1}(t)-\log D_{k+1}(t)
=-\log\bigl(1-\beta_k(t)^2\bigr).
\]
Its difference from $-\log(1-\Delta/(2t))$ is at most
$2\cdot10^{26}/\Delta^2$.
For the explicit approximation $W$ in~\eqref{der:detmodels}, use
the second-difference formula~\eqref{pre:peano},
$\partial_k^2V=\log(k/(2t))$, $\partial_k^4V=-1/k^2$, and the two
logarithmic-prefactor bounds $1/[12(k-1)^2]$ and
$1/[8(\Delta-1)^2]$. The fixed-radius curvature error of $E$ is at most
$4\cdot10^{26}/\Delta^2$; adding \eqref{der:eq:Bcurv} bounds that of
$E+B$ by $1004\cdot10^{26}/\Delta^2$.
Moving the two neighboring radii from $t$ to $\tau_j$ costs at most
$300\cdot10^{26}/t$ for $E$ and $480000/t$ for $B$, by their derivative
bounds.

The remaining explicit saddle expression, up to a $k$-independent constant, is
\[
 nH(k/\sqrt n)-\frac1{12}\log k
 -\frac18\log\frac{4n-k^2}{4n+k^2}
 -\frac12\log\frac{n+k^2/4}{2},
\]
with $H$ from \eqref{low:H}. Direct differentiation gives $-H''=Q$ and
$|Q''|<100$ on $[1,2]$, which contains the neighboring arguments.
Thus action discretization costs at most $100/(12n)\le17/t^2$.
For $j\in[k-1,k+1]$, $j\ge t$ and $2n/j-j/2\ge\Delta/2$, whence
\[
 \frac{d^2}{dj^2}\left[-\frac18\log\frac{4n-j^2}{4n+j^2}\right]
 =\frac{2n(3j^4+16n^2)}{(j^2-4n)^2(j^2+4n)^2}\le6/\Delta^2.
\]
The other logarithmic prefactors contribute at most $1/(12t^2)$ and
$1/t^2$. Their combined error is therefore at most
$10/\Delta^2+100/t^2$.
Since
$
 24\cdot10^{56}+1004\cdot10^{26}+10<3\cdot10^{57},
 \ 300\cdot10^{26}+480000<10^{30},
$
these bounds and \eqref{left:eq:etaCurv} prove the first inequality in
\eqref{left:eq:leftresult}.

Finally $Q(k/\sqrt n)\ge\Delta/(2t)$. The domain assumptions give
\[
 \Delta^3/t\ge8\cdot10^{60}>8(3\cdot10^{57}),\qquad
 \Delta\ge2\cdot10^{40}>16\cdot10^{30},\qquad
 \Delta t\ge2\cdot10^{100}>16(3\cdot10^{57}).
\]
Thus the $\Delta^{-2}$, $t^{-1}$ and $t^{-2}$ errors cost at most
$\Delta/(8t)$, $\Delta/(16t)$ and $\Delta/(16t)$, respectively.
This proves \eqref{left:eq:leftresult}.

For \eqref{left:eq:leftwindow}, write $c=k/\sqrt n$. Then
\[
 t=\sqrt n(1/c+c/4),\qquad
 \Delta=(2\sqrt n-k)\frac{c+2}{2c},\qquad
 \frac\Delta t=\frac{2(4-c^2)}{4+c^2}<\frac85\,10^{-13}.
\]
The identity
$(c+2)^3-2c^2(c^2+4)=(2-c)(2c^3+3c^2+8c+4)\ge0$
shows $\Delta/t^{1/3}\ge(2\sqrt n-k)n^{-1/6}=T$.
Hence the stated scaled window implies the saddle hypotheses, completing
the proof of Lemma~\ref{left:thm:left}.
\end{proof}

The regional estimates now yield the main result.

\begin{proof}[Proof of Theorem~\ref{thm:main}]
Take $T_0$ as in~\eqref{left:eq:leftwindow} and put $M:=T_0+1$.
Since $M+2<(201/100)10^{20}$ and $\log10>9/4$,
\[
 25000(M+2)^3
 <\frac{(201/100)^3}{4}\,10^{65}
 <\frac94\,10^{65}<\log N.
\]
As $N$ is an integer, the compact threshold satisfies $N_c(M)\le N$.
Also $N>10^{300000}$, so all other regional starting values are met.
In particular $n^{-1/6}<1$.
The endpoint coefficients $P_{n,1}=P_{n,n}=1$ give strict inequalities
because the outside-support neighbors vanish. For an interior index,
we use the following exhaustive alternatives.
\begin{enumerate}
\item If $k\le(2-10^{-13})\sqrt n$, Lemma~\ref{low:theorem} applies.
\item Otherwise, if $T\ge T_0$, then $k<2\sqrt n$ and
Lemma~\ref{left:thm:left} applies.
\item If $T<T_0$ and $s\le23$, then
$-T_0-1/(2n^{1/6})<s\le23$, hence $|s|<M$.
Lemma~\ref{compact:thm:compact} applies.
\item If $s>23$ and $k<2\sqrt n+10n^{1/6}(\log n)^{2/3}$, then
$23<s<10(\log n)^{2/3}<11(\log n)^{2/3}$, so Lemma~\ref{right:theorem} applies.
\item At or above that upper boundary, Lemma~\ref{up:theorem} applies.
\end{enumerate}
Each regional statement includes the neighboring values in its
log-concavity conclusion; the neighbors need not separately satisfy its
stated central-index condition. The alternatives cover every supported
index and establish the theorem.
\end{proof}

\section{Concluding remarks}\label{sec:concluding}

Theorem~\ref{thm:main} gives an effective eventual form of Chen's
log-concavity conjecture: for every $n\ge10^{10^{65}}$, the sequence
$(P_{n,k})_{k=1}^{n}$ is strictly log-concave throughout its support.
The proof combines exact determinantal identities with quantitative
estimates in five overlapping regions. The essential step is to
control the errors at fixed size more sharply than the corresponding
discrete curvature, since log-concavity of a limiting or Poissonized
distribution alone does not imply the desired inequalities.

The explicit threshold is far too large for the remaining range
to be checked directly. Thus a complete proof of
Conjecture~\ref{conj:chen} requires either substantially sharper
uniform estimates or an argument that avoids a large threshold.
Improving the coefficient-extraction bounds and the estimates in
the transition regions may reduce the threshold. A complementary
approach is to seek a combinatorial or algebraic proof, for example
through the Schur-positivity framework of Gao, Xie and
Yang~\cite{GXY}.

Chen~\cite[Conjectures~1.2--1.4]{Chen} also proposed three further
conjectures concerning these numbers. Besides the infinite
log-concavity of each sequence $(P_{n,k})_{k=1}^{n}$, he conjectured
the strong $q$-log-convexity and the infinite $q$-log-convexity of
the polynomial sequence $(P_n(q))_{n\ge1}$, where
\[
P_n(q):=\sum_{k=1}^{n}P_{n,k}q^k.
\]
To the best of our knowledge, these three conjectures remain open
in general. Partial evidence is available: B\'ona, Lackner and
Sagan~\cite[\S\S6.3--6.4]{BLS} established infinite log-concavity
for $n\le50$ and verified ordinary $q$-log-convexity through that
range.

The present theorem addresses first-order log-concavity for
sufficiently large $n$; it does not establish these stronger
properties. Higher-order log-concavity requires control under
repeated log-concavity transforms, while the polynomial conjectures
require coefficientwise inequalities linking different sizes.
Developing methods that handle these additional requirements is
a natural direction for further work.

\vspace{10mm}
\noindent{\bf Acknowledgments.}
This work was supported by the National Natural Science Foundation
of China (Grant No.~12371334) and the Qinglan Project of Jiangsu Province. 

\medskip
\noindent{\bf Competing Interests.}
The author declares no competing interests.

\medskip
\noindent{\bf Data availability.}
 Data sharing is not applicable to this article as no datasets were generated
 or analyzed during the current study.

\medskip
\noindent{\bf Declaration of AI usage.}
During the development of this work, the author used AI as a research-assistance tool to explore possible proof
strategies and to assist with language editing and proofreading.
All mathematical arguments presented in the paper were independently
verified, refined, and written by the author, who takes full
responsibility for the final content of the publication.

\end{document}